\documentclass[reqno, 10pt]{amsart}
\usepackage{amssymb}
\usepackage[colorlinks=true]{hyperref}
\usepackage[margin=1in]{geometry}
\usepackage{enumitem}

\newenvironment{enum}
{\begin{enumerate}[leftmargin=0pt]
  \setlength{\itemsep}{2pt}
  \setlength{\parskip}{0pt}
  \setlength{\parsep}{0pt}}
{\end{enumerate}}

\newenvironment{itm}
{\begin{itemize}[leftmargin=0pt]
  \setlength{\itemsep}{2pt}
  \setlength{\parskip}{0pt}
  \setlength{\parsep}{0pt}}
{\end{itemize}}

\newtheorem{theorem}{Theorem}[section]
\newtheorem{lemma}[theorem]{Lemma}
\newtheorem{proposition}[theorem]{Proposition}
\newtheorem{corollary}[theorem]{Corollary}
\theoremstyle{remark}
\newtheorem{observation}[theorem]{Remark}
\newtheorem{assumption}{Assumption}

\newtheorem{definition}{Definition}[section]

\newcommand{\les}{\lesssim}
\newcommand{\dd}{\,d}

\newcommand{\R}{\mathbb R}
\renewcommand{\H}{\mathbb H}
\newcommand{\C}{\mathbb C}

\newcommand{\lb}{\label}
\newcommand{\be}{\begin{equation}}
\newcommand{\ee}{\end{equation}}

\newcommand{\mc}{\mathcal}
\newcommand{\B}{\mc B}
\newcommand{\K}{\mc K}

\newcommand{\M}{\mc M}
\newcommand{\D}{\mc D}

\newcommand{\ov}{\overline}
\newcommand{\ds}{\displaystyle}
\DeclareMathOperator{\re}{Re}
\DeclareMathOperator{\im}{Im}

\DeclareMathOperator{\erfc}{erfc}
\DeclareMathOperator{\sgn}{sgn}

\DeclareMathOperator{\Ker}{Ker}
\renewcommand{\ker}{\Ker}
\DeclareMathOperator{\supp}{supp}

\title[Spectral multipliers II]{Spectral multipliers II: elliptic and parabolic operators and Bochner--Riesz means}
\author{Marius Beceanu}
\author{Michael Goldberg}
\subjclass{35J08, 35J10, 35J25, 35K10, 35K15, 37J11, 42B08, 42B15, 42B37, 47A25, 47D60}

\begin{document}
\maketitle

\begin{abstract} We establish estimates for the Poisson kernel, the heat kernel, and Bochner--Riesz means defined in terms of $H=-\Delta+V$, where $V$ is an arbitrarily large rough real-valued scalar potential and $H$ can have negative eigenvalues. All results are in three space dimensions.

We eliminate several unnecessary conditions on $V$, leaving just $V \in \K_0$, where $\K_0$ is the closure of the space of test functions in the global Kato class $\K$, meaning
$$
\sup_{y \in \R^3} \int_{\R^3} \frac {|V(x)| \dd x}{|x-y|} < \infty.
$$

For the spectral multiplier bounds, we assume that $H$ has no zero or positive energy bound states. For $V \in \K_0$, we prove that $H$ has at most a finite number of negative bound states. If in addition $V \in \dot W^{-1/4, 4/3}$, then by \cite{golsch} and \cite{kota} there are no positive energy bound states.

\end{abstract}

\tableofcontents

\section{Introduction}
\subsection{Main results} Consider Hamiltonians of the form $H=-\Delta+V$ in $\R^3$. Here $-\Delta$ is the free Hamiltonian, corresponding to free motion, while $V$ is a scalar real-valued (i.e.~self-adjoint) potential. Most results carry over to the non-selfadjoint case (complex or matrix-valued), which will be treated elsewhere.

Take $V$ in the (global) Kato class introduced in \cite{rodsch}, meaning that its global Kato norm
\be\lb{katonorm}
\|V\|_\K := \sup_{y \in \R^3} \int_{\R^3} \frac{|V(x)| \dd x}{|x-y|} < \infty
\ee
is finite. Under a similar local condition, $H=-\Delta+V$ is self-adjoint, see \cite{simon}. Note that $\K \subset \dot H^{-1}_{loc}$ and, acting by multiplication, $\K \subset \B(\dot H^1, \dot H^{-1})$.

Further assume that $V \in \K_0=\ov{\D}_\K$, the Kato class closure of the space of test functions $\D = C^\infty_c$. The space $\K_0$ contains the Lorentz space $L^{3/2, 1}$, see \cite{bergh} for more on this topic.

Throughout the paper we also make the following spectral assumption:
\begin{assumption}\lb{assum} Zero is a regular point of the spectrum of $H$, meaning that there are no zero energy eigenstates or resonances, meaning that $I+V (-\Delta)^{-1} \in \B(\M)$ is invertible.
\end{assumption}

Our main results are the following:
\begin{proposition}[Proposition \ref{poisson}, Poisson kernel bounds] Assume that $V \in \K_0$ and $H$ has no zero or positive energy bound states. If in addition $H=-\Delta+V$ has no negative energy bound states, the perturbed Poisson kernel $e^{-t\sqrt H}$ is dominated by the free Poisson kernel:
$$
|e^{-t\sqrt H}(x, y)| \les \frac t {(t^2 + |x-y|^2)^2}.
$$
If $H$ has negative eigenvalues, then the kernel of $e^{-t\sqrt H} P_c$ is the sum of two terms, $K_1(t)$ and $K_2(t)$, such that $K_1(t)$ is dominated by the free Poisson kernel and $K_2(t)$ is dominated by a $t$-independent kernel $K_{20} \in \B(L^p)$, $1 \leq p \leq \infty$, in a manner that also ensures uniform decay as $t \to \infty$:
$$
|K_2(t)(x, y)| \les \langle t \rangle^{-1} K_{20}(x, y).
$$
Furthermore, as $t \to 0$, $K_2(t)$ converges strongly to $-P_p$ in $\B(L^p)$, $1 \leq p \leq \infty$, and for $t \geq 1$ $K_2(t)$ is also dominated by the free Poisson kernel.

Finally, the following bounds hold for the perturbed Poisson kernel: for $t \leq 1$
$$
|e^{-t\sqrt H}(x, y)| \les \frac t {(t^2 + |x-y|^2)^2}
$$
and for $t \geq 1$
$$
|e^{-t\sqrt H}P_c(x, y)| \les \frac t {(t^2 + |x-y|^2)^2}.
$$
\end{proposition}
Similar bounds can be proved for all $t$ derivatives of the Poisson kernel, as well as for one antiderivative, $\frac {e^{-t \sqrt H}}{\sqrt H}$.

\begin{proposition}[Proposition \ref{heat}, heat kernel bounds] Assume that $V \in \K_0$ and $H$ has no zero or positive energy bound states. If in addition $H$ has no negative energy bound states, then the perturbed heat kernel $e^{-tH}$ is dominated by the free heat kernel:
$$
|e^{-tH}(x, y)| \les t^{-3/2} e^{-\frac{|x-y|^2}{4t}}.
$$
If $H$ has negative eigenvalues, then $e^{-tH}P_c(x, y)$ can be decomposed into two parts, $K_1(t)$ and $K_2(t)$, such that for every $\epsilon>0$ $K_1(t)$ satisfies the same Gaussian bound with a constant of size $\epsilon^{-1}$, and $K_2(t)$ is dominated by a $t$-independent, exponentially decaying kernel $K_{2\epsilon} \in \B(L^p)$, $1 \leq p \leq \infty$, in a manner that also ensures uniform decay as $t \to \infty$:
$$
|K_2(t)(x, y)| \les t^{-1} e^{-(\frac {\sqrt{-\lambda_1}} 2 - \epsilon) t} K_{2\epsilon}(x, y).
$$
Here $\lambda_1$ is the (negative) eigenvalue of $H$ closest to $0$ and $K_{2\epsilon}$ decays exponentially in $x$ and $y$.

Finally, for $t \leq 1$, 
$$
|e^{-tH}(x, y)| \les t^{-3/2} e^{-\frac {|x-y|^2}{4t}}.
$$
The same is true for all $t>0$, but with an exponentially growing factor of $e^{-\lambda_N t}$ instead of $t^{-3/2}$, where $\lambda_N<0$ is the smallest eigenvalue of $H$.
\end{proposition}
Similar bounds can be proved for all $t$ derivatives of the heat kernel, as well as for one antiderivative, $\frac {e^{-tH}}{H}$.

Consider the maximal function
$$
[Mf](x) = \sup_{x \in B(y, R)} \frac 1 {|B(y, r)|} \int_{B(y, R)} |f(z)| \dd z.
$$
It is known that $M$ is bounded on $L^p$, $1<p \leq \infty$ and on $BMO$ and is of weak-$(1, 1)$ type. The previous two results, together with some basic facts about the bound states of $H$, imply that
\begin{corollary} Assuming that $V \in \K_0$ and $H=-\Delta+V$ has no zero or positive energy bound states,
$$
\sup_{t > 0} |e^{-t\sqrt H} f| \les Mf,\ \sup_{t \in (0, 1]} |e^{-t H} f| \les Mf,\ \sup_{t >0} |e^{-t H} P_c f| \les Mf.
$$
\end{corollary}
This further implies the norm convergence of the solution to the initial data, meaning that $e^{-t\sqrt H}$ and $e^{-tH}$ are $C_0$ semigroups.

\begin{proposition}[Proposition \ref{BR}, Bochner--Riesz means] Assume that $V \in \K_0$ and $H=-\Delta+V$ has no zero or positive energy bound states. Then $(I - H/\lambda_0)_+^\alpha P_c$ has the same $L^p$ boundedness properties as the fractional integration operator $J_{(1-a)/2}$, where $a=\re \alpha$: $(I - H/\lambda_0)_+^\alpha P_c \in \B(L^p, L^q)$ if $\frac 1 p - \frac 1 q \geq \frac {1-a} 3$ and $-1<a<1$.

In addition, $(I - H/\lambda_0)_+^\alpha P_c \in \B(L^p)$ whenever $|\frac 1 p - \frac 1 2|<\frac a 2$, $p \in [1, \infty]$.
\end{proposition} 

Our new spectral results in this paper, see below, imply that all previous results from \cite{becgolschr} and \cite{becgol} hold supposing only that $V \in \K_0$ and that $H=-\Delta+V$ has no zero energy bound states (Assumption \ref{assum}), removing the need for any further spectral or regularity conditions.

About half the results in \cite{becche} and \cite{becgol3} are also true in this generality. The others also depend on certain $L^p$ bounds for the intertwining operators $H^s(-\Delta)^{-s}$ and $(-\Delta)^s H^{-s}$; see our upcoming paper for more details.

\subsection{The Kato class}
The potential $V \in \K$ is a short-range potential of finite Kato norm
$$
\|V\|_\K := \sup_{y \in \R^3} \int_{\R^3} \frac{|V(x)| \dd x}{|x-y|} < \infty,
$$
as defined in \cite{rodsch} and \cite{golsch}. The Kato norm is translation-invariant and has the same scaling invariance as $-\Delta$:
$$
\|V(\alpha x)\|_{\K_x} = \alpha^{-2} \|V\|_\K.
$$
Other scaling-invariant norms for $-\Delta+V$ are e.g.\ those of $L^{3/2}$ or in general $L^{3/2, q}$, see \cite{bergh} about the Lorentz spaces $L^{p, q}$. Note that $L^{3/2, 1} \subset \K$ and $L^{3/2, 1} \subset L^{3/2, q}$, but $L^{3/2, q} \not \subset \K$ for $q > 1$.

Strichartz estimates and uniform resolvent bounds hold when $V \in L^{3/2, \infty}$, but many other estimates require $V \in \mc K$. For $V$ of less than quadratic decay, estimates are no longer uniform on the whole continuous spectrum, only away from the threshold.


The Kato class $\K$ includes some singular continuous measures (of locally bounded variation) $V \in \M_{loc}$, if we interpret its definition as
$$
\|V\|_{\K} = \sup_{y \in \R^3} \int_{\R^3} \frac {d|V(x)|}{|x-y|} < \infty.
$$

Any such measures must be continuous, but need not be absolutely continuous. Clearly, the Kato norm of a probability measure supported on a set is a measure of the capacity of that set. By (\ref{frostman}), $V$ is a Frostman measure of dimension at least $1$, meaning that the Hausdorff dimension of its support is at least $1$. Neither decay at infinity, see \cite{bec3} (some sort of lacunarity is required instead), nor local integrability, see \cite{goldberg}, is always necessary. Some results hold in general when $V \in \K$.

Nevertheless, in this paper we assume that $V \in \K_0$, so it decays in norm at infinity and it is locally integrable, hence its support has full Hausdorff dimension.

More precisely, as for $VMO \subset BMO$, $\K_0$ consists of those measures in $\mc K$ having three additional properties: local integrability and

\begin{definition}[The local Kato property]
$$
\lim_{\epsilon \to 0} \sup_{y \in \R^3} \int_{|x-y|<\epsilon} \frac{|V(x)| \dd x}{|x-y|} = 0.
$$
\end{definition}

\begin{definition}[The modified distal Kato property]
$$
\lim_{R \to \infty} \sup_{y \in \R^3} \int_{|x|>R} \frac{|V(x)| \dd x}{|x-y|} = \lim_{R \to \infty} \|\chi_{|x|>R} V(x)\|_\K = 0.
$$
\end{definition}

Despite appearances, the modified distal Kato property is translation-invariant. For comparison, we also state the distal Kato property:
\begin{definition}[The distal Kato property]
$$
\lim_{R \to \infty} \sup_{y \in \R^3} \int_{|x-y|>R} \frac{|V(x)| \dd x}{|x-y|} = 0.
$$
\end{definition}

One property implies the other:
\begin{lemma}[Lemma \ref{modkato}] For $V \in \K$, the modified distal Kato property implies the distal Kato property.
\end{lemma}

The converse is not true: for example, a potential made of identical bumps situated at dyadic (or even polynomially increasing) distances from the origin satisfies the distal Kato condition, but not the modified one.

Many of the results in this paper hold assuming only the local and modified distal Kato properties. However, for convenience, we also assume local integrability. The case when $V$ is a singular measure will be treated elsewhere.

By means of local integrability and the local and modified distal Kato properties, we completely characterize $\K_0$.

\begin{lemma}[Lemma \ref{charkato}] For a measure $V \in \K$, $V \in \K_0$ if and only if $V$ has the local and modified distal Kato properties and $V \in L^1_{loc}$.
\end{lemma}
A more detailed discussion and the proof follow in Appendix \ref{a}.

Finally, if $L^{3/2, 1} \subset \K$, then by duality $\K^* \subset L^{3, \infty}$. Furthermore, since $\K$ is a lattice, so is $\K^*$. A more precise characterization, stated without proof in \cite{becgol}, is as follows:
\begin{proposition}[Proposition \ref{katodual}] $g \in \K^*$ if and only if there exists a positive measure $\mu$, with $\|\mu\|_\M=\mu(\R^3)=\|g\|_{K^*}$, such that $|g| \leq (-\Delta)^{-1} \mu$ almost everywhere.
\end{proposition}
We also prove this characterization in Appendix \ref{a}.

\subsection{Spectral properties of the Hamiltonian. Bound states}

If $V \in \K$, then $\sigma_c(H) = \sigma_{ac}(H) = [0, \infty)$. However, in this paper, $V$ can be negative. A major difference when $V$ is large and negative is the possible presence of bound states.

The continuous and point spectrum have to be treated separately. Under our assumptions the point spectrum bounds are trivial to prove (or impossible, in the case of dispersive estimates), while the continuous spectrum bounds are non-trivial, but can be proved under Assumption \ref{assum}.

All previous results in \cite{becgolschr}, \cite{becgol}, to some extent in \cite{becche} and \cite{becgol3}, and the current ones depend exclusively on the following facts about $H=-\Delta+V$ and its bound states:

\begin{enum}
\item The spectrum of $H$ consists of finitely many negative eigenvalues $\lambda_k$, for which the corresponding eigenfunctions $f_k$ decay exponentially:
$$
f_k(x) \les \frac {e^{-\sqrt{-\lambda_k}|x|}}{\langle x \rangle}
$$

\item In particular, $H$ has no bound states at zero energy, so  $I+VR_0(\lambda)$ is uniformly invertible along $\sigma_c(H)=[0, \infty)$. 

\item Conditions for Wiener's theorem: Let $S(t)(x, y)= \chi_{t \geq 0}\frac 1 {4 \pi t} V(x) \delta_t(|x-y|)$. Then
$$
\lim_{R \to \infty} \|\chi_{\rho \geq R} S(\rho)\|_{L^\infty_y \M_x L^1_\rho} = 0 \text{ (decay at infinity)}
$$
and for some $N \geq 1$
$$
\lim_{\epsilon \to 0} \|S^N(\rho+\epsilon) - S^N(\rho)\|_{L^\infty_y \M_x L^1_\rho} = 0 \text{ (continuity under translation)},
$$
where powers correspond to repeated convolutions.
\end{enum}

The first two conditions are rather natural spectral assumptions. Except for the absence of zero energy bound states, they hold whenever $V \in \D$.

Zero energy bound states are absent generically, as shown below, but not for all $V \in \D$. The presence of zero energy bound states changes the analysis and leads to different decay rates in decay estimates, due to a more complicated interaction with the continuous spectrum; see \cite{ersc}, \cite{yaj}, \cite{zero}, and \cite{bec2}, among others, for more details. In this paper we confine ourselves to the generic case, so we have to specifically impose Assumption \ref{assum}.

In Appendix \ref{b} we prove that the first two conditions are true (absent zero energy bound states) under the mere assumption that $V \in \K_0$ or less than that, since local integrability is not needed:
\begin{proposition}[Proposition \ref{finite}] If $V \in \K_0$ is real-valued (self-adjoint), then $H$ has the same number of negative eigenvalues as the number of (real) eigenvalues of $V (-\Delta)^{-1} \in \B(\M) \cap \B(\K)$ in $(-\infty, -1)$, which is finite.
\end{proposition}
In the course of the proof in Appendix \ref{b}, we also show that $H_t=-\Delta+tV$ only has zero energy bound states for discretely many values of $t$, providing one sense in which zero energy bound states are non-generic.

\begin{proposition}[Proposition \ref{no_pos}] If $V \in \K_0$ is real-valued (self-adjoint), then $H=-\Delta+V$ has no embedded resonances in $(0, \infty)$ and embedded eigenvalues in $(0, \infty)$ are discrete. If in addition $V \in L^{3/2, \infty}$ or $V \in \dot W^{-1/4, 4/3}$, then there are no embedded eigenvalues.
\end{proposition}

Thus, we need at least one extra condition, such as $V \in L^{3/2, \infty}$ (or even $V \in \dot W^{-1/4, 4/3}$), under which the absence of embedded eigenvalues is known by the work of \cite{ioje} and \cite{kota}.


If the zero energy bound states have enough decay --- if they are in $L^1$, see \cite{zero} and \cite{bec2} ---, then they decouple from the continuous spectrum. In general, though, in three dimensions, zero energy resonances are of size $\langle x \rangle^{-1}$ and zero energy $L^2$ eigenfunctions can be of size $\langle x \rangle^{-n}$, for $n \geq 2$. There can be up to $13=1+3+3^2$ obstructions to decoupling, see \cite{bec2}, which make Assumption \ref{assum} necessary.

Finally, the conditions for Wiener's theorem are stable under passing to the limit in $\K$ (so the sets of $V \in \K$ for which each condition is true are closed in $\K$), so they are true for any $V \in \K_0$ because they are true for any $V \in \D$:
\begin{proposition}[Proposition \ref{cond_wiener}] Consider $V \in \K$ and let $S(t)(x, y)= \chi_{t \geq 0}\frac 1 {4 \pi t} V(x)$ $\delta_t(|x-y|)$. If $V$ fulfills the distal Kato condition, then $S$ fulfills the decay at infinity condition:
$$
\lim_{R \to \infty} \|\chi_{\rho \geq R} S(\rho)\|_{L^\infty_y \M_x L^1_\rho} = 0.
$$
If $V \in \K_0$, then $S^N$ fulfills the continuity under translation condition for $N=2$:
$$
\lim_{\epsilon \to 0} \|S^2(\rho+\epsilon) - S^2(\rho)\|_{L^\infty_y \M_x L^1_\rho} = 0.
$$
\end{proposition}

\subsection{Spectral multipliers} By functional calculus, one can define a family of commuting operators $f(H)$, such that $f(H)=H$ when $f(z)=z$. These are called spectral multipliers and are initially defined for (germs of) analytic functions $f$ bounded on $\sigma(H)$, the \textbf{spectrum} of $H$, by a Cauchy integral around it. More generally, since $H$ is self-adjoint, $f(H)$ can be defined in terms of the spectral measure
$$
f(H) = \int_{\sigma(H)} f(\lambda) \dd E_H(\lambda)
$$
and is $L^2$-bounded whenever $f$ is a bounded Borel measurable function on $\sigma(H)$. Either way, $f \mapsto f(H)$ is an algebra homomorphism.

Common classes of spectral multipliers include the resolvent, fractional integration operators, Mihlin multipliers, Bochner--Riesz sums, the Schr\"{o}dinger evolution, the wave and half-wave propagators, the heat flow, the Poisson kernel, and others, see below.

Some of these operators were studied in \cite{becgol3}, while others are the subject of our present paper.

If $V$ is a small perturbation of $-\Delta$, then $H$ behaves similarly to $-\Delta$ and, more specifically, $f(H)$ has properties similar to those of $f(-\Delta)$. Our goal is to extend these properties to the case of large perturbations, including when $H$ has bound states.

In the free case ($V=0$) in $\R^3$, the following properties are well-known:

\begin{enum}
\item Resolvent bounds: The free resolvent $R_0(\lambda) := (-\Delta-\lambda)^{-1}$ is analytic on $\C \setminus \sigma(-\Delta)$, where $\sigma(-\Delta)=[0, \infty)$, having the integral kernel
$$
R_0(\lambda)(x, y) = \frac 1 {4\pi} \frac {e^{-\sqrt{-\lambda}}|x-y|}{|x-y|}. 
$$
Here $\sqrt z$ is the main branch of the square root.

By the limiting absorption principle, $R_0(\lambda)$ has continuous (in e.g.\ the strong $\B(L^{6/5}, L^{6})$ topology, not in norm) extensions to $[0, \infty)$, where it approaches different boundary values on either side:
$$
R_0(\lambda \pm i0)(x, y) = \frac 1 {4\pi} \frac {e^{\pm i\sqrt\lambda |x-y|}}{|x-y|}.
$$

Along the real line, $R_0$ is not $L^2$-bounded, but is still uniformly bounded from $L^{6/5}$ to $L^6$ or (stronger) from $L^{6/5, 2}$ to $L^{6, 2}$. Furthermore, it obeys the uniform bound, for $\lambda \in \C$,
$$
|R_0(\lambda)(x, y)| \leq \frac 1 {4\pi |x-y|}.
$$

Such resolvent bounds were obtained in \cite{becgol3} for the perturbed case.

\item Fractional integration operators: For $0<\re s<3/2$, $I_s=(-\Delta)^{-s}$ has the integral kernel
$$
I_s(x, y) = C_s |x-y|^{2s-3},
$$
while for $\re s >0$ the kernel of $J_s=(I-\Delta)^{-s}$ has the same decay at infinity and is also bounded.

Such pointwise bounds and corresponding $L^p$ fractional integration estimates were obtained for the perturbed case in \cite{becgol3}.

\item Schr\"{o}dinger and wave propagators: the free Schr\"{o}dinger and wave propagators
satisfy dispersive estimates such as pointwise decay, Strichartz, and reversed Strichartz estimates, as proved in \cite{becgolschr}, \cite{becgol}, and \cite{becgol3}.

Dispersive estimates for the Klein--Gordon propagator are addressed in \cite{becche}.

\item Heat and Poisson kernels: For $t>0$, the free heat kernel is given by the Gaussian
$$
e^{t\Delta}(x, y) := \frac 1 {(4\pi t)^{3/2}} e^{-\frac{|x-y|^2}{4t}}
$$
and forms an approximation to the identity.

More generally, the perturbed heat kernel also satisfies Gaussian bounds, including when $H$ has bound states, as we prove in this paper.

The free Poisson kernel is given for $t>0$ by
$$
e^{-t\sqrt{-\Delta}}(x, y) = \frac 1 {4\pi^2} \frac t {(|x-y|^2+t^2)^2}
$$
and also forms an approximation to the identity. In general, the perturbed Poisson kernel is bounded by a multiple of the free Poisson kernel, as we prove here.

\item Mihlin spectral multipliers: Several ways of defining the Mihlin symbol class lead to substantially similar properties. Mihlin's classical condition is that $|m^{(k)}(\lambda)| \les \lambda^{-k}$ for $0 \leq k \leq \lfloor d/2 \rfloor+1$. A sharper condition due to H\"{o}rmander is
$$
\sup_{\alpha > 0} \|\chi(\lambda) m(\lambda/\alpha)\|_{H^{d/2+\epsilon}} < \infty
$$
for any standard cutoff function $\chi$.


In the free case, Mihlin multipliers $m(H)$ are bounded on $L^p$, $1<p<\infty$, and of weak $(1, 1)$ type, i.e.\ bounded from $L^1$ to $L^{1, \infty}$. Mihlin multiplier bounds for the perturbed case are addressed in \cite{becgol3}.

\item Bochner--Riesz means: Bochner--Riesz means correspond to the spectral multipliers $M_\alpha(\lambda) = (1-\lambda/\lambda_0)_+^\alpha$. In the free case one can take $\lambda_0=1$, due to the scaling symmetry. The value $\alpha=0$ corresponds to the Dirichlet kernel, $\alpha=d$ to Ces\`{a}ro summation of the Fourier series, and $\alpha<0$ to restriction theorems.

In the free case, the integral kernel of $(I+\Delta)_+^\alpha$ is
$$
(I+\Delta)_+^\alpha(x, y)=\frac {\Gamma(\alpha+1) J_{d/2+\alpha}(2\pi|x-y|)} {\pi^\alpha |x-y|^{d/2+\alpha}}.
$$

Consequently, $(I+\Delta)_+^\alpha$ satisfies at least the same $L^p$ bounds as the fractional integration operator $J_{(1-\alpha)/2}$, but more hold due to the oscillations of the Bessel functions. For example, $(I+\Delta)_+^\alpha$ is $L^2$-bounded whenever $\re \alpha \geq 0$.

\cite{fef} showed that above dimension $1$, when $\alpha=0$ $\chi_{[0, 1]}(-\Delta) \in \B(L^p)$ only when $p=2$. In general, $M_\alpha \in \B(L^p)$ whenever $|\frac 1 p - \frac 1 2| < \frac {\alpha}{d-1}= \frac {\alpha} 2$ for $d=3$, but $M_\alpha \not \in \B(L^p)$ when $|\frac 1 p - \frac 1 2| \geq \frac {2\alpha+1}{2d}= \frac {2\alpha+1} 6$. Various bounds are true in between.

Here we extend some of these $L^p$ bounds to the perturbed case.
\end{enum}

Several of these multiplier classes form algebras, semigroups parametrized over $[0, \infty)$ or e.g.~$\{0\} \cup \{z \in \C \mid \im z > 0\}$, or groups. Most of them are $L^2$-bounded, because their symbols are bounded on $\sigma(-\Delta)=[0, \infty)$.

Spectral multipliers can be defined simply by spectral calculus, but must be radially symmetric. More general multipliers, such as partial derivatives or non-radial Mihlin multipliers including Riesz transforms, correspond to non-radial symbols and are not defined by spectral calculus, but by means of the (skewed) Fourier transform or wave operators. They will be the subject of a different paper.

\subsection{Discussion of the results} 

Heat kernel bounds are affected by the presence of any bound states, including negative energy ones. However, see Proposition \ref{heat}, which establishes heat kernel bounds in the presence of negative energy bound states.

One way to guarantee that there are no bound states is to assume that $V \geq 0$ or at least its negative part $V_-:=\min(V, 0)$ has a sufficiently small Kato norm (\ref{katonorm}):
$$
\|V_-\|_\K < 4\pi.
$$
If $V \geq 0$, one can weaken the regularity and decay assumptions on $V$, a common assumption then being that $V$ belongs to a reverse H\"{o}lder class.

Absent bound states, one can use heat kernel estimates to retrieve several inequalities on our list, see \cite{sikora}. Gaussian heat kernel bounds lead to estimates for Bochner-Riesz sums, Riesz transforms, certain Mihlin multipliers, etc.; see \cite{jiang}. Gaussian heat kernel bounds are known on more general Riemannian manifolds, see \cite{grigoryan}, or metric measure spaces, see \cite{sturm}, \cite{duo}, \cite{ams}, \cite{jiang}.

Our results are new in that they also apply to Hamiltonians with bound states and hold for a wider range of Lebesgue spaces.

%
%


\section{Notations and preliminary remarks}

\noindent Notations:
\begin{itm}
\item $a \les b$ means that there exists a constant $C$ such that $|a| \leq C |b|$
\item $\langle x \rangle=(1+|x|^2)^{1/2}$
\item $a_+=\max(a, 0)$
\item $\dot H^s$, $H^s$, $\dot W^{s, p}$, $W^{s, p}$ are Sobolev spaces
\item $\M$ is the Banach space of measures of bounded variation
\item $R_0(\lambda)=(-\Delta-\lambda)^{-1}$ and $R_V(\lambda)=(-\Delta+V-\lambda)^{-1}$ are the free and the perturbed resolvents
\item $AB$ is the image of the Banach space $B$ under the operator $A$, which is again a Banach space; an example is $\langle x \rangle \M$, where $\langle x \rangle$ means pointwise multiplication by $\langle x \rangle$
\item $A(\lambda)=VR_0(\lambda)$, $A_t(\lambda)=I+tA(\lambda)$.
\end{itm}

When $V \in \K$ has the local and modified distal properties, by our present results and by Assumption \ref{assum} the point spectrum of $H$ consists of $N$ negative eigenvalues $\lambda_1\geq\ldots\geq\lambda_N$ with associated eigenfunctions $f_k$.

We have the following Agmon-type bound for the decay of eigenfunctions:
\begin{proposition} Under the previous assumptions, let $f \in L^2$ be a solution of the eigenfunction equation
$$
H f = \lambda f
$$
for $\lambda<0$. Then
$$
|f(x)| \les \frac {e^{-\sqrt{-\lambda} |x|}}{\langle x \rangle}.
$$
\end{proposition}

Therefore
\be\lb{eigen}
|f_k(x) \otimes \ov f_k(y)| \les \frac {e^{-\sqrt{-\lambda_k}(|x|+|y|)}}{\langle x \rangle \langle y \rangle}.
\ee
A proof of this result is contained in \cite{becgol}.

Multiplier bounds are all related, for a fixed Hamiltonian $H$, and each can be used to prove several others. Since it is the most singular, having good control on the forward wave propagator
$$
T(\tau) = \chi_{\tau \geq 0} \frac {\sin(\tau\sqrt H) P_c}{\sqrt H}
$$
(see \cite{becgolschr} and \cite{becgol}) implies most of the other bounds.

An important consideration is the finite speed of propagation for solutions to the wave equation, which leads to sharp estimates for $T$ outside the light cone: it vanishes or decreases exponentially in space. In fact, there its integral kernel has the explicit form
\be\lb{explicit}
\chi_{|x-y|>\tau} T(\tau)(x, y) = - \chi_{|x-y|>\tau} \sum_{k=1}^N \frac{\sinh(\tau \sqrt{-\lambda_k})}{\sqrt{-\lambda_k}} f_k(x) \otimes \ov f_k(y).
\ee

Indeed, due to its finite speed of propagation, the wave evolution for $H$ vanishes outside the light cone. In case $H$ has no bound states, the same is true for the integral kernel of $T$, which equals the full evolution then. Otherwise, outside the light cone the contribution of the continuous spectrum exactly cancels out that of the point spectrum, but neither vanishes. This gives the precise expression (\ref{explicit}).

Inside the light cone, the behavior of $T$ is governed by two bounds from \cite{becgol} and \cite{becgol3}, where the former is especially important:
$$
\sup_{x, y} \|\tau T(\tau)(x, y)\|_{\M_\tau} < \infty,\ \|T(\tau)(x, y)\|_{\M_\tau} \les \frac 1 {|x-y|}.
$$
While these hold both inside and outside the light cone, outside (\ref{explicit}) is sharper.

\section{Poisson kernel bounds}

Poisson's equation on a domain $D$ is
$$
-\Delta f=0 \text{ on } D,\ f=g \text{ on } \partial D.
$$
On unbounded domains one may need to impose extra conditions to make the solution unique. On the upper half-space $\H^{n+1}=\{(x, t) \mid x \in \R^n, t>0\}$, a standard condition is that the solution should vanish at infinity (otherwise it can be non-unique up to multiples of $t$).

The solution can then be expressed by integrating a Green's function against the boundary data $g$.

For Poisson's equation in the upper half-space, Green's function is the free Poisson kernel on the upper half-space $\H^{n+1}$, $e^{-t\sqrt{-\Delta}}$, $t>0$, and when $n=3$ is given by
$$
e^{-t\sqrt{-\Delta}}(x, y) = \frac 1 {\pi^2} \frac t {(t^2 + |x-y|^2)^2}.
$$

The perturbed Poisson kernel $e^{-t\sqrt H}$ for the Hamiltonian $H=-\Delta+V$, where $V$ is $t$-independent, similarly acts as Green's function for the perturbed Poisson's equation
$$
(-\Delta+V) f=0 \text{ on } \H^{3+1},\ f=g \text{ on } \R^3=\partial \H^{3+1}.
$$
The negative energy/point spectrum component $e^{-t\sqrt H} P_p$ is given by the finite-rank operator
$$
e^{-t\sqrt H} P_p = \sum_{k=1}^N e^{-t\sqrt{\lambda_k}} f_k(x) \otimes \ov f_k(y),
$$
where $\sqrt{\lambda_k} \in i\R$ since $\lambda_k<0$. Either branch of the square root leads to complex oscillations, but no decay in $t$.

We next characterize the positive energy/continuous spectrum component $e^{-t\sqrt H} P_c$.
\begin{proposition}[The perturbed Poisson kernel]\lb{poisson} Assume that $V \in \K_0$ and $H$ has no zero or positive energy bound states. If in addition $H=-\Delta+V$ has no negative energy bound states, the perturbed Poisson kernel $e^{-t\sqrt H}$ is dominated by the free Poisson kernel:
\be\lb{bound1}
|e^{-t\sqrt H}(x, y)| \les \frac t {(t^2 + |x-y|^2)^2}.
\ee
If $H$ has negative eigenvalues, then the kernel of $e^{-t\sqrt H} P_c$ is the sum of two terms, $K_1(t)$ and $K_2(t)$, such that $K_1(t)$ is dominated by the free Poisson kernel and $K_2(t)$ is dominated by a $t$-independent kernel $K_{20} \in \B(L^p)$, $1 \leq p \leq \infty$, in a manner that also ensures uniform decay as $t \to \infty$:
\be\lb{bound2}
|K_2(t)(x, y)| \les \langle t \rangle^{-1} K_{20}(x, y).
\ee
Furthermore, as $t \to 0$, $K_2(t)$ converges strongly to $-P_p$ in $\B(L^p)$, $1 \leq p \leq \infty$, and for $t \geq 1$ $K_2(t)$ is also dominated by the Poisson kernel.

Finally, the following bounds hold for the perturbed Poisson kernel: for $t \leq 1$
\be\lb{bound3}
|e^{-t\sqrt H}(x, y)| \les \frac t {(t^2 + |x-y|^2)^2}
\ee
and for $t \geq 1$
\be\lb{bound4}
|e^{-t\sqrt H}P_c(x, y)| \les \frac t {(t^2 + |x-y|^2)^2}.
\ee
\end{proposition}

\begin{proof}
The positive energy component of the Poisson kernel is given by the formula
\be\lb{repr}\begin{aligned}
e^{-t\sqrt H} P_c (x, y) &= \frac 1 {2\pi i} \int_0^\infty e^{-t\sqrt \lambda} (R_V(\lambda+i0) - R_V(\lambda-i0)) \dd \lambda \\
&= \frac 1 {\pi i} \int_{-\infty}^\infty \lambda e^{-t|\lambda|} R_V((\lambda+i0)^2) \dd \lambda \\
&= \frac 4 \pi \int_0^\infty \frac {t\tau} {(t^2+\tau^2)^2} T(\tau)(x, y) \dd \tau.
\end{aligned}\ee

For $\delta<1$, we distinguish the two regions $\tau \geq \delta |x-y|$, inside and near the light cone, and $0 \leq \tau < \delta |x-y|$, strictly outside the light cone, whose contributions we call $K_1(t)$, respectively $K_2(t)$.

In the first region, we get a bound of
$$
\bigg|\int_{|x-y|/2}^\infty \frac {t\tau} {(t^2+\tau^2)^2} T(\tau)(x, y) \dd \tau\bigg| \leq \sup_{\tau \geq \delta |x-y|} \frac t {(t^2+\tau^2)^2} \sup_{x, y} \|\tau T(\tau)\|_{\M_\tau} \les \frac t {(t^2+\delta^2|x-y|^2)^2},
$$
so $K_1(t)$ is dominated by the Poisson kernel.

In the second region $\tau<\delta|x-y|$, the same computation only yields a bound of $t^{-3}$ or $t^{-2} |x-y|^{-1}$ with more effort. To improve on that, we use the explicit form (\ref{explicit}) of the integral kernel of $T$ in that region.
%

In case $H$ has no bound states, the integral kernel of $T$ cancels outside the light cone. Thus, in this case there is no contribution from the second region and we already have the correct bound in the first region (and we can take $\delta=1$), which proves (\ref{bound1}).

Otherwise, the expression (\ref{explicit}) has exponential growth in $\tau$, but we only evaluate it over the finite interval $\tau \in [0, \delta |x-y|)$, on which the exponential growth is compensated by the exponential decay in space of the eigenstates $f_k$.

Replace $T(\tau)$ by (\ref{explicit}) in the expression of $K_2(t)$:
\be\lb{k2}
K_2(t)(x, y) = -\frac 4 \pi \sum_{k=1}^N \bigg(\int_0^{\delta|x-y|} \frac {t\tau} {(t^2+\tau^2)^2} \frac{\sinh(\tau \sqrt{-\lambda_k})}{\sqrt{-\lambda_k}} \dd \tau\bigg) f_k(x) \otimes \ov f_k(y).
\ee
For each $1 \leq k \leq N$ we estimate 
\be\lb{complicated}
\int_0^{\delta|x-y|} \frac {t\tau} {(t^2+\tau^2)^2} \frac{\sinh(\tau \sqrt{-\lambda_k})}{\sqrt{-\lambda_k}} \dd \tau.
\ee

To prove Poisson bounds, it suffices to compare an expression to $t^{-3}$ and to $t |x-y|^{-4}$ separately. Furthermore, when $t \leq 1$, for the first comparison it suffices to prove the expression is bounded.

It is easy to see that the contribution of (\ref{complicated}) is dominated by $t^{-3}$. On the other hand, we shall see that it does not vanish at $t=0$, so $K_2(t)$ does not satisfy the second bound.

The expression (\ref{complicated}) is roughly bounded by
$$
\bigg(\int_0^{\delta |x-y|} \frac {t \dd\tau} {t^2+\tau^2}\bigg) \sup_{\tau \leq \delta |x-y|} \frac{\sinh(\tau \sqrt{-\lambda_k})}{\tau \sqrt{-\lambda_k}} \les \min\bigg(1, \frac {\delta |x-y|} t\bigg) \cosh(\delta|x-y| \sqrt{-\lambda_k}),
$$
since
$$
\frac {\sinh x} x \leq \cosh x,\ \frac {\cosh x - 1}{x^2} \leq \frac {\cosh x} 2.
$$
Taking into account the decay of $f_k(x) \otimes \ov f_k(y)$ (\ref{eigen}) in the full expression of $K_2(t)$,
$$
|f_k(x) \otimes \ov f_k(y)| \les \frac {e^{-\sqrt{-\lambda_k}(|x|+|y|)}}{\langle x \rangle \langle y \rangle},
$$
we get a uniform bound over the second region, for the term corresponding to the $k$-th bound state $f_k$, of
$$
\min(1, \delta/t) e^{-\sqrt{-\lambda_k}(|x|+|y|-\delta|x-y|)}
$$
where $\min(1, \delta/t)$ can be replaced by $\langle t \rangle^{-1}$ for fixed $\delta<1$, say $\delta=1/2$.

For $\delta=1$, this integral kernel and the related one
$$
\frac {e^{-\sqrt{-\lambda_k}(|x|+|y|-\delta|x-y|)}}{\langle x \rangle \langle y \rangle}
$$
are not $L^p$-bounded. However, they are $L^p$-bounded when $\delta<1$.

This proves (\ref{bound2}) when $H$ has bound states, with
$$
K_{20}(x, y)=\sum_{k=1}^N e^{-(1-\delta)\sqrt{-\lambda_k}(|x|+|y|)}.
$$
Together with the $t^{-3}$ bound, it also implies Poisson estimates for $K_2(t)$ when $t \geq 1$, proving (\ref{bound4}).


To obtain more precise $K_2(t)$ asymptotics as $t \to 0$, assume $t \leq 1$. Integrating by parts in (\ref{complicated}) yields
\be\lb{inte}
\int_0^{\delta|x-y|} \frac t {2(t^2+\tau^2)} \cosh(\tau \sqrt{-\lambda_k}) \dd \tau -\frac t {2(t^2+\delta^2|x-y|^2)} \frac {\sinh(\delta|x-y|\sqrt{-\lambda_k})}{\sqrt{-\lambda_k}}.
\ee
Note that
$$
\frac 1 {t^2 + |x-y|^2}
$$
is, up to a constant, the integral kernel of $\frac {e^{-t\sqrt{-\Delta}}}{\sqrt {-\Delta}}$.

The second (boundary) term in (\ref{inte}) satisfies Poisson bounds, since it is dominated by
$$
\frac t {\delta^2 |x-y|^2} e^{-(1-\delta)\sqrt{-\lambda_k}(|x|+|y|)} \les t \delta^{-2} |x-y|^{-4}
$$
and also
$$
\frac {t \delta |x-y|} {2(t^2+\delta^2|x-y|^2)} \frac {\sinh(\delta|x-y|\sqrt{-\lambda_k})}{\delta |x-y| \sqrt{-\lambda_k}} \les e^{-(1-\delta)\sqrt{-\lambda_k}(|x|+|y|)} \leq 1.
$$
As $t$ approaches $0$, the first term in (\ref{inte}) converges to $\pi/4$. We estimate the difference as follows:
$$\begin{aligned}
&\int_0^{\delta|x-y|} \frac t {2(t^2+\tau^2)} \cosh(\tau \sqrt{-\lambda_k}) \dd \tau - \frac \pi 4 = \\
&=\int_{\delta |x-y|}^\infty \frac {t \dd \tau}{2(t^2+\tau^2)} + \int_0^{\delta|x-y|} \frac {t \tau^2} {2(t^2+\tau^2)} \frac {\cosh(\tau \sqrt{-\lambda_k}) - 1} {\tau^2} \dd \tau \\
&\les \min\bigg(1, \frac t {\delta |x-y|}\bigg) + t\delta|x-y| \cosh(\delta|x-y| \sqrt{-\lambda_k}) |\lambda_k|.
\end{aligned}$$
Reintroducing $f_k(x) \otimes \ov f_k(y)$, we see that this expression also satisfies Poisson bounds.

Summing over $k$, the difference between the previous expression and $\frac \pi 4 P_p$ satisfies Poisson bounds, uniformly for $t \leq 1$.

After including the factors of $-1$ from (\ref{explicit}) and $4/\pi$ from (\ref{repr}), we find that the difference between $K_2(t)$ and $-P_p$ satisfies Poisson bounds. We then replace $P_p$ by $e^{-t\sqrt H} P_p$, since their difference is also dominated by the free Poisson kernel. Since the same is true for $K_1(t)$, the whole perturbed Poisson kernel $e^{-t\sqrt H}$ is dominated by the free Poisson kernel when $t \leq 1$, proving (\ref{bound3}).

Another way to evaluate the difference is by dominated convergence. Both the second term in (\ref{inte}) and the difference between the first term in (\ref{inte}) and $\pi/4$ vanish, because they are dominated by
$$
\min\bigg(1, \frac t {\delta |x-y|}\bigg) \cosh(\delta|x-y|\sqrt{-\lambda_k})
$$
and then by
$$
\min\bigg(1, \frac t {\delta |x-y|}\bigg) \frac {e^{-(1-\delta)\sqrt{-\lambda_k}(|x|+|y|)}}{\langle x \rangle \langle y \rangle}
$$
after including $f_k(x) \otimes \ov f_k(y)$. This implies strong convergence to $0$ in $\mc B(L^p)$, $1 \leq p \leq \infty$, as $t$ approaches $0$.

After including the neglected factors of $-1$ and $4/\pi$ and summing over $k$, this dominated convergence results in the strong convergence of $K_2(t)$ to $-P_p$ in $\mc B(L^p)$, $1 \leq p \leq \infty$, finishing the proof.
\end{proof}


\section{Heat kernel bounds}

The free heat kernel in three dimensions has the formula
$$
e^{t\Delta}(x, y) = \frac 1 {(4\pi t)^{3/2}} e^{-\frac {|x-y|^2}{4t}},
$$
where $t>0$. Like the Poisson kernel, the heat kernel is an approximation to the identity.

In general, the perturbed heat kernel can grow exponentially due to negative energy bound states:
$$
e^{-tH} P_p(x, y) = \sum_{k=1}^N e^{-\lambda_k t} f_k(x) \otimes \ov f_k(y),
$$
where $\lambda_k<0$. However, after projecting these bound states away, one can obtain finer estimates for the remaining part
$$
e^{-tH} P_c = \frac 1 {2\pi i} \int_0^\infty e^{-t\lambda} (R_V(\lambda+i0) - R_V(\lambda-i0)) \dd \lambda.
$$

\begin{proposition}[Heat kernel bounds]\lb{heat} Assume that $V \in \K_0$ and $H$ has no zero or positive energy bound states. If in addition $H$ has no negative energy bound states, then the perturbed heat kernel $e^{-tH}$ is dominated by the free heat kernel:
\be\lb{gauss1}
|e^{-tH}(x, y)| \les t^{-3/2} e^{-\frac{|x-y|^2}{4t}}.
\ee
If $H$ has negative eigenvalues, then $e^{-tH}P_c(x, y)$ can be decomposed into two parts, $K_1(t)$ and $K_2(t)$, such that, for every $\epsilon>0$, $K_1(t)$ satisfies the same Gaussian bound (\ref{gauss1}) with a constant of size $\epsilon^{-1}$ and $K_2(t)$ is dominated by a $t$-independent, exponentially decaying kernel $K_{2\epsilon} \in \B(L^p)$, $1 \leq p \leq \infty$, in a manner that also ensures uniform decay as $t \to \infty$:
\be\lb{gauss2}
|K_2(t)(x, y)| \les t^{-1} e^{-(\frac {\sqrt{-\lambda_1}} 2 - \epsilon) t} K_{2\epsilon}(x, y).
\ee
Here $\lambda_1$ is the eigenvalue of $H$ closest to $0$.

Finally, for $t \leq 1$, 
$$
|e^{-tH}(x, y)| \les t^{-3/2} e^{-\frac {|x-y|^2}{4t}}.
$$
The same is true for all $t>0$, but with an exponentially growing factor of $e^{-\lambda_N t}$ instead of $t^{-3/2}$, where $\lambda_N$ is the eigenvalue of $H$ farthest from $0$.
\end{proposition}


\begin{proof}
The continuous spectrum component of the heat kernel is given by
$$\begin{aligned}
e^{-tH} P_c (x, y)&= \frac 1 {2\pi i} \int_0^\infty e^{-t\lambda} (R_V(\lambda+i0) - R_V(\lambda-i0)) \dd \lambda \\
&= \frac 1 {\pi i} \int_{-\infty}^\infty \lambda e^{-t\lambda^2} R_V((\lambda+i0)^2) \dd \lambda\\
&= \frac 1 {2 \sqrt \pi t^{3/2}} \int_0^\infty \tau e^{-\frac{\tau^2}{4t}} T(\tau) (x, y) \dd \tau.
\end{aligned}$$
We again distinguish two regimes: $\tau \geq |x-y|$ and $0 \leq \tau < |x-y|$, corresponding to $K_1(t)$ and $K_2(t)$ respectively. In the first region
$$
t^{-3/2} \sup_{\tau \geq |x-y|} e^{-\frac{\tau^2}{4t}} \sup_{x, y} \|\tau T(\tau)(x, y)\|_{\M_\tau} \les t^{-3/2} e^{-\frac{|x-y|^2}{4t}}.
$$

If there are no bound states then $T$ vanishes in the second region, strictly outside the light cone, so there is no contribution from it. Hence we obtain exactly the desired bound (\ref{gauss1}) for this case: the perturbed heat kernel is dominated by the free heat kernel.

If $H$ has bound states, then in the second region, outside the light cone, $T(\tau)$ is given by the explicit formula (\ref{explicit}).
For each $k$ we need to estimate
\be\lb{comp2}
t^{-3/2} \int_0^{|x-y|} \tau e^{-\frac{\tau^2}{4t}} \frac{\sinh(\tau \sqrt{-\lambda_k})}{\sqrt{-\lambda_k}} \dd \tau.
\ee

We do it as follows:
$$\begin{aligned}
(\ref{comp2})&=\frac 2 {t^{1/2} \sqrt{-\lambda_k}} \int_{-\frac {|x-y|}{2t^{1/2}}}^{\frac {|x-y|}{2t^{1/2}}} \tau e^{-\tau^2} e^{2\tau t^{1/2} \sqrt{-\lambda_k}}\dd \tau\\
&= \frac {2 e^{-\lambda_k t}} {t^{1/2} \sqrt{-\lambda_k}} \int_{-t^{1/2}\sqrt{-\lambda_k}-\frac {|x-y|}{2t^{1/2}}}^{-t^{1/2}\sqrt{-\lambda_k} + \frac {|x-y|}{2t^{1/2}}} (\tau + t^{1/2} \sqrt{-\lambda_k}) e^{-\tau^2} \dd \tau.
\end{aligned}$$
We distinguish two regimes, according to whether $-t^{1/2}\sqrt{-\lambda_k} + \frac {|x-y|}{2t^{1/2}}$ is negative (in which case there are further gains) or not. This is equivalent to asking whether $|x-y|<2t\sqrt{-\lambda_k}$ or not.

Assuming $|x-y|< 2\delta t\sqrt{-\lambda_k}$, where $\delta<1$, we use the following approach: on the integration interval $|\tau + t^{1/2} \sqrt{-\lambda_k}| \leq \frac {|x-y|}{t^{1/2}}$, then we estimate the integral using $\erfc$, the complementary error function. Then (\ref{comp2}) is bounded by
$$
\frac {e^{-t\lambda_k}} {t^{1/2} \sqrt{-\lambda_k}} \frac {|x-y|}{t^{1/2}} \erfc(t^{1/2}\sqrt{-\lambda_k} - \frac {|x-y|}{2t^{1/2}}) \les (1-\delta)^{-1} t^{-3/2} (-\lambda_k)^{-1} e^{-\frac {|x-y|^2}{4t}} e^{|x-y|\sqrt{-\lambda_k}} |x-y|,
$$
as $\erfc a \les e^{-a^2}/a$. Considering that
$$
|f_k(x) \otimes \ov f_k(y)| \les \frac {e^{-\sqrt{-\lambda_k}(|x|+|y|)}}{\langle x \rangle \langle y \rangle},
$$
we obtain a sharp Gaussian bound for $K_2(t)$ in this region.

On the other hand, when $|x-y|>2\delta t \sqrt{-\lambda_k}$, we can only bound the integral by $\frac {|x-y|}{t^{1/2}}$, $\frac {|x-y|^2}{t}$, or similar. Then
$$
(\ref{comp2}) \les t^{-1} e^{-\lambda_k t} (-\lambda_k)^{-1/2} |x-y|.
$$
Clearly, in this regime we cannot obtain Gaussian bounds, because the eigenfunctions only have exponential decay in $x$ and (\ref{comp2}) does not provide it either.

However, reintroducing $f_k(x) \otimes \ov f_k(y)$, we get a bound of
$$
t^{-1} e^{-\lambda_k t - \sqrt{-\lambda_k}(|x|+|y|)} \les t^{-1} e^{-(1-\frac 1 {2\delta})\sqrt{-\lambda_k}(|x|+|y|)}.
$$
Taking $\delta>1/2$ and summing over $k$ we obtain a uniform bound for large $t$, in terms of the $L^p$-bounded operator
$$
K_{20}=e^{-(1-\frac 1 {2\delta})\sqrt{-\lambda_1}(|x|+|y|)},
$$
where $\lambda_1$ is the eigenvalue of $H$ closest to $0$.

Giving up some exponential decay in $x$ and $y$, it is easy to see that $t^{-1}$ can be replaced by $t^{-1} e^{-(\frac {\sqrt{-\lambda_1}} 2-\epsilon) t}$ for any $\epsilon>0$, proving (\ref{gauss2}).

Another way of evaluating (\ref{comp2}), which yields a more precise estimate for small $t$, is by computing the integral in (\ref{comp2}) on $[0, \infty)$, then estimating the difference.

The integral on $[0, \infty)$ leads to nothing but $-e^{-tH} P_p$, after summing over $k$, while the difference is
\be\lb{comp3}
\frac 1 {2t^{3/2} \sqrt{-\lambda_k}} \int_{\R \setminus (-|x-y|, |x-y|)} \tau e^{-\frac {\tau^2}{4t}} e^{\tau \sqrt{-\lambda_k}} \dd \tau.
\ee
Since $-e^{-tH} P_p$ grows exponentially and $K_2(t)$ is uniformly bounded, their difference (\ref{comp3}) will also grow with $t$.

By successive changes of variable (\ref{comp3}) becomes
$$\begin{aligned}
\frac {2 e^{-\lambda_k t}} {t^{1/2} \sqrt{-\lambda_k}} \int_{\R \setminus (-t^{1/2}\sqrt{-\lambda_k}-\frac {|x-y|}{2t^{1/2}}, -t^{1/2}\sqrt{-\lambda_k} + \frac {|x-y|}{2t^{1/2}})} (\tau + t^{1/2} \sqrt{-\lambda_k}) e^{-\tau^2} \dd \tau.
\end{aligned}$$
This is bounded by
\be\lb{inte2}
\frac {e^{-\lambda_k t}} {t^{1/2} \sqrt{-\lambda_k}} (e^{-(\frac {|x-y|}{2t^{1/2}}-t^{1/2}\sqrt{-\lambda_k})^2} + t^{1/2} \sqrt{-\lambda_k} \erfc(\frac {|x-y|}{2t^{1/2}}-t^{1/2}\sqrt{-\lambda_k})).
\ee
Simplifying the first term in (\ref{inte2}) leads to $(-\lambda_k)^{-1/2} t^{-1/2} e^{-\frac {(x-y)^2}{4t}} e^{|x-y|\sqrt{-\lambda_k}}$. This satisfies Gaussian bounds, if we include $f_k(x) \otimes \ov f_k(y)$. However, we got a factor of $t^{-1/2}$ instead of $t^{-3/2}$, which is better for small $t$.

On the other hand, $\erfc$ in the second term in (\ref{inte2}) can be as large as $1$, in the region where $|x-y| \leq 2t \sqrt{-\lambda_k}$.

Assuming that $\frac {|x-y|} {2t^{1/2}} - t^{1/2} \sqrt{-\lambda_k} \geq \delta t^{1/2} \sqrt{-\lambda_k}$, where $\delta>0$, or equivalently that $|x-y| \geq 2(1+\delta)t\sqrt{-\lambda_k}$, the second term satisfies the same Gaussian bound as the first term (up to a factor of $\delta^{-1}$), as $\erfc a \les e^{-a^2}/a$.

We are then left with the region where $|x-y|<2(1+\delta)t\sqrt{-\lambda_k}$, in which $\erfc$ can be as large as $1$. Here the second term is bounded by $e^{-\lambda_k t}$.

However, in this region $\frac {|x-y|} {2(1+\delta)t} < \sqrt{-\lambda_k}$, so $\frac {|x-y|^2}{2(1+\delta) t} < \sqrt{-\lambda_k} |x-y|$ and
$$
e^{-\frac {|x-y|^2}{2(1+\delta)t}} > e^{-\sqrt{\lambda_k}|x-y|} \geq e^{-\sqrt{\lambda_k}(|x|+|y|)}.
$$
Taking into account $f_k(x) \otimes \ov f_k(y)$ and (\ref{eigen}), the contribution of the second term in (\ref{inte2}) in this region is then bounded by $e^{-\lambda_k t} e^{-\frac {|x-y|^2}{2(1+\delta)t}} \langle x \rangle^{-1} \langle y \rangle^{-1}$.

In other words, this expression satisfies Gaussian bounds in space, for a better Gaussian (roughly the square of the original), but with a factor of $e^{-\lambda_k t}$ (as claimed) instead of $t^{-3/2}$. This is better when $t$ is small and worse when $t$ is large.

Due to the estimates we obtained for (\ref{comp3}), which are better than Gaussian for small $t$, the difference $K_2(t) -(-e^{-tH}P_p)$ goes to $0$ as $t \to 0$.
\end{proof}

\section{Bochner--Riesz means}

Bochner--Riesz means correspond to the spectral multipliers $M_\alpha(\lambda) = (1-\lambda/\lambda_0)_+^\alpha = \max(1-\lambda/\lambda_0, 0)^\alpha$. In the free case, due to the scaling symmetry, $\|(I+\Delta/\lambda_0)_+^\alpha\|_{\B(L^p)}= 1$ is constant with respect to $\lambda_0$ when this operator is bounded, but not in the perturbed case.

In general, the point spectrum component is given by
$$
(\lambda_0-H)_+^\alpha P_p= \sum_{k=1}^N (1-\lambda_k/\lambda_0)^\alpha f_k \otimes \ov f_k.
$$
Its $B(L^p)$ norm is bounded when $\lambda_0 \to +\infty$, but blows up as $\lambda_0 \to 0$.

The exponent $\alpha$ need not be real; in fact, we need $\alpha \in \C$ for complex interpolation.

\begin{proposition}[Bochner--Riesz means]\lb{BR} Assume that $V \in \K_0$ and $H$ has no zero or positive energy bound states. Then $(I - H/\lambda_0)_+^\alpha P_c$ has the same $L^p$ boundedness properties as the fractional integration operator $J_{(1-a)/2}$, where $a=\re \alpha$: $(I - H/\lambda_0)_+^\alpha P_c \in \B(L^p, L^q)$ if $\frac 1 p - \frac 1 q \geq \frac {1-a} 3$ and $-1<a<1$.

In addition, $(I - H/\lambda_0)_+^\alpha P_c \in \B(L^p)$ whenever $|\frac 1 p - \frac 1 2|<\frac a 2$, $p \in [1, \infty]$.
\end{proposition}
\begin{proof}
Recall the following facts about Bessel functions: for $a=\re \alpha>-1$
$$\begin{aligned}
\mc F^{-1} (1-\lambda^2)_+^\alpha &= \frac {\sqrt{2\pi} 2^{\alpha} \Gamma(\alpha+1) J_{\alpha+1/2}(\tau)}{\tau^{\alpha+1/2}},\\
\mc F^{-1} (1-\lambda^2/\lambda_0)_+^\alpha &= \frac {\sqrt{2\pi \lambda_0} 2^{\alpha} \Gamma(\alpha+1) J_{\alpha+1/2}(\lambda_0^{1/2}\tau)}{(\lambda_0^{1/2} \tau)^{\alpha+1/2}},\\
|J_\alpha(\tau)| &\les \min(\tau^{a}, e^{\frac \pi 2 |\im \alpha|} \tau^{-1/2}),
\end{aligned}$$
and for any $\alpha \in \C$ $\partial_\tau (J_\alpha(\tau)/\tau^\alpha) = -J_{\alpha+1}(\tau)/\tau^\alpha$. We shall ignore the $e^{\frac \pi 2 |\im \alpha|}$ factor, except to note that it behaves well under complex interpolation.

The continuous spectrum component of the Bochner-Riesz sum is given by
$$\begin{aligned}
(I-H/\lambda_0)_+^\alpha P_c (x, y)&= \frac 1 {2\pi i} \int_0^\infty (\lambda_0-\lambda)_+^\alpha (R_V(\lambda+i0) - R_V(\lambda-i0)) \dd \lambda \\
&= \frac 1 {\pi i} \int_{-\infty}^\infty \lambda (\lambda_0-\lambda^2)_+^\alpha R_V((\lambda+i0)^2) \dd \lambda\\
&= \sqrt{2\pi \lambda_0} 2^{\alpha+1} \Gamma(\alpha+1) \int_0^\infty \frac {J_{\alpha+3/2}(\lambda_0^{1/2} \tau)}{(\lambda_0^{1/2} \tau)^{\alpha+1/2}} T(\tau) (x, y) \dd \tau.
\end{aligned}$$

The integral splits into two regions, according to whether $\delta |x-y| \leq \tau$ or $0 \leq \tau < \delta|x-y|$, where we fix $\delta<1$. In the first region
$$\begin{aligned}
\bigg|\lambda_0^{1/2} \int_{\delta|x-y|}^\infty \frac {J_{\alpha+3/2}(\lambda_0^{1/2} \tau)}{(\lambda_0^{1/2}\tau)^{\alpha+1/2}} T(\tau) (x, y) \dd \tau \bigg|
&\leq \lambda_0 \sup_{\delta |x-y| \leq \tau} \frac {J_{\alpha+3/2}(\lambda_0^{1/2}\tau)}{(\lambda_0^{1/2} \tau)^{a+3/2}} \sup_{x, y} \|\tau T(\tau)(x, y)\|_{\M_\tau} \\
&\les \lambda_0 \min(1, (\lambda_0^{1/2} \delta |x-y|)^{-2-a}).
\end{aligned}$$
This expression can be further bounded by e.g. $\lambda_0^{-a/2} \delta^{-2-a} |x-y|^{-2-a}$, so is in $L^{\frac 3 {2+a}, \infty} \cap L^\infty$ with a norm that depends on $\lambda_0$. Another possibility is
$$
\lambda_0 \min(1, (\lambda_0^{1/2} \delta |x-y|)^{-2-a}) \les \lambda_0 (\lambda_0^{1/2} \delta |x-y|)^{-2} = \delta^{-2} |x-y|^{-2},
$$
independently of $\alpha$ and $\lambda_0$, when $a \geq 0$.

Finally, another bound is
$$
\lambda_0^{3/2} \min((\lambda_0^{1/2} \delta |x-y|)^{-1}, (\lambda_0^{1/2} \delta |x-y|)^{-2-a}),
$$
which is a rescaled version (leaving the $L^1$ norm invariant) of $\min((\delta |x-y|)^{-1}, (\delta |x-y|)^{-2-a})$, which is in $L^1$ when $a>1$.

In the second region, outside the light cone, $T$ is given by (\ref{explicit}). For each $k$ we evaluate
\be\lb{comp_BR}
\lambda_0^{1/2} \int_0^{\delta|x-y|} \frac {J_{\alpha+3/2}(\lambda_0^{1/2} \tau)}{(\lambda_0^{1/2} \tau)^{\alpha+1/2}} \frac {\sinh(\tau \sqrt{-\lambda_k})}{\sqrt{-\lambda_k}} \dd \tau.
\ee
A rough estimate gives
$$
(\ref{comp_BR}) \les \lambda_0 \sup_{\tau \leq \delta|x-y|} \frac {J_{\alpha+3/2}(\lambda_0^{1/2} \tau)}{(\lambda_0^{1/2} \tau)^{\alpha+3/2}} \sup_{\tau \leq \delta|x-y|} \frac {\tau \sinh(\tau \sqrt{-\lambda_k})}{\sqrt{-\lambda_k}} \les \lambda_0 (-\lambda_k)^{-1/2} \delta|x-y| e^{\delta |x-y| \sqrt{-\lambda_k}}
$$
on one hand and
$$\begin{aligned}
(\ref{comp_BR}) &\les \lambda_0^{-a/2} \sup_{\tau \leq \delta|x-y|} |(\lambda_0^{1/2} \tau)^{1/2} J_{\alpha+3/2}(\lambda_0^{1/2} \tau)| \int_0^{\delta|x-y|} \frac {\sinh(\tau \sqrt{-\lambda_k})}{\tau^{\alpha+1} \sqrt{-\lambda_k}} \dd \tau \\
&\les \lambda_0^{-a/2} |x-y|^{1-a} e^{\delta|x-y| \sqrt{-\lambda_k}}
\end{aligned}$$
for $a<1$ on the other hand. These two powers of $\lambda_0$ are the two extremes we got when evaluating the contribution of the first region.

Between these two extremes, for $a \geq 0$
$$\begin{aligned}
(\ref{comp_BR}) &\les \sup_{\tau \leq \delta|x-y|} \bigg|\frac {J_{\alpha+3/2}(\lambda_0^{1/2} \tau)} {(\lambda_0^{1/2} \tau)^{\alpha-1/2}}\bigg| \int_0^{\delta|x-y|} \frac {\sinh(\tau \sqrt{-\lambda_k})}{\tau \sqrt{-\lambda_k}} \dd \tau \les |x-y| e^{\delta|x-y| \sqrt{-\lambda_k}}
\end{aligned}$$
independently of $\lambda_0$.

Reintroducing the exponentially decaying $f_k(x) \otimes \ov f_k(y)$ factor bounded by (\ref{eigen}), overall the contribution of the second region is bounded by $e^{-\sqrt{-\lambda_k}(|x|+|y|-\delta|x-y|)}$. This integral kernel is bounded on all Lebesgue spaces.

If $-1<a<1$, the integral kernel of $(I-H/\lambda_0)_+^\alpha P_c$ is in $L^{\frac 3 {2+a}, \infty} \cap L^\infty$, so it satisfies fractional integration bounds. If $a>1$, we obtain that the kernel is bounded by an integrable function of $x-y$, so it is $L^p$-bounded, $1 \leq p \leq \infty$. Note that the bound does not depend on $\lambda_0$. By interpolation with the line $a=0$, where the operator is $L^2$-bounded, we obtain the general conclusion.
\end{proof}

\appendix
\section{The Kato class}\lb{a}
Here we include some proofs of properties of the Kato class $\K$ and of $\K_0=\ov{\D}_\K$.

\subsection{Local integrability and the local and distal Kato properties}

\begin{lemma}\lb{modkato} For $V \in \K$, the modified distal Kato property implies the distal Kato property.
\end{lemma}
\begin{proof} Assume that $V$ satisfies the modified distal Kato property and let $R_0$ be a radius such that $\|\chi_{|x|>R_0} V(x)\|_\K < \epsilon$.

By dominated convergence,
$$
\lim_{y \to \infty} \int_{|x|\leq R_0} \frac {|V(x)| \dd x}{|x-y|} = 0.
$$
It follows that for all sufficiently large $y$, say $|y|>R_1$,
$$
\int_{\R^3} \frac {|V(x)| \dd x}{|x-y|} < 2\epsilon.
$$
Finally, for $|y|<R_1$ and any $R>R_1+R_0$
$$
\int_{|x-y|>R} \frac {|V(x)| \dd x}{|x-y|} \leq \|\chi_{|x|>R_0} V(x)\|_\K < \epsilon.
$$
Therefore, for any $y$ and any $R>R_1+R_0$ the integral is less than $2\epsilon$, meaning that it vanishes uniformly in $y$ as $R \to \infty$.
\end{proof}

For $V \in \K$,
\be\lb{frostman}
\int_{|x-y| \leq R} |V(x)| \dd x \leq R \|V\|_\K,
\ee
though this condition is not sufficient for $V$ to be in $\K$.

In addition, if $V$ has the local Kato property then
$$
\lim_{R \to 0} R^{-1} \int_{|x-y| \leq R} |V(x)| \dd x = 0
$$
and if $V$ also has the distal Kato property then the same is true as $R \to \infty$. This further justifies the comparison to $VMO$ versus $BMO$.

A partial converse is as follows: let
$$
g(R)=R^{-1} \sup_y \int_{|x-y| \leq R} |V(x)| \dd x,
$$
which is bounded for $V \in \K$. Conversely, if $g(R)/R \in L^1(0, \infty)$ or equivalently $g(2^k) \in \ell^1_k$, then $V \in \K$ and it has the local Kato property.

This sufficient, but not necessary, condition shows that $V \in \K$ need not belong to $L^p_{loc}$ with $p>1$. On the other hand, $\K_0 \subset L^1_{loc}$. This raises the issue of the singular measures in $\K$ --- which, although not belonging to $\K_0$, may still be in some wider suitable class of potentials.

By definition, the Kato class is included in $\Delta L^\infty \subset \dot W^{-2, BMO}$, which is a non-separable Banach space of tempered distributions. We consider the spaces $\Delta C_0$ and $(\Delta L^\infty)_0=\ov{\D}_{\Delta L^\infty}$, the closure of the class of test functions in $\dot W^{-2, \infty}$, which are separable. Convergence in $\dot W^{-2, \infty}$ is weaker than norm convergence for $\K$.

\begin{lemma} $\ov \D_{\Delta L^\infty} \subset \Delta C_0$. If $V \geq 0$, $V \in \K \cap \Delta C_0$ is equivalent to $V \in \K \cap \ov{\D}_{\Delta L^\infty}$.
\end{lemma}
\begin{proof} Let $\chi$ be a standard cutoff function.
%
%

If $V$ is approximated by $V_k \in C^\infty$ in $\Delta L^\infty$, then $\Delta^{-1} V$ is approximated by $\Delta^{-1} V_k \in C^\infty \subset C$ in $L^\infty$. But $C$ is closed in $L^\infty$, so $V \in \Delta C$. A similar argument shows that if $V$ is approximated by $V_k \in \D$ then $V \in \Delta C_0$.

Next, we consider decay at infinity, which requires extra conditions in one direction. If $(-\Delta)^{-1} V$ vanishes at infinity, then for any $\epsilon>0$ there exists $R_0$ such that
$$
\sup_{|y|>R_0} \bigg|\int_{\R^3} \frac {V(x) \dd x}{|x-y|}\bigg| < \epsilon.
$$
If $V \in \K$ and $|y|>R$, then
$$
\int_{\R^3} \frac {\chi_{|x| \leq R} |V(x)| \dd x}{|x-y|} \leq \frac 1 {|y|-R} \|\chi_{|x| \leq R} V(x)\|_\M \leq \frac R {|y|-R} \|V\|_K.
$$
Then for sufficiently large $|y|>R_1=R(1+\epsilon^{-1}\|V\|_\K)$
$$
\sup_{|y|>R_1} \int_{\R^3} \frac {\chi_{|x| \leq R} |V(x)| \dd x}{|x-y|} < \epsilon,
$$
so for $|y|>\max(R_0, R_1)$
$$
\sup_{|y|>\max(R_0, R_1)} \bigg|\int_{\R^3} \frac {\chi_{|x|>R} V(x) \dd x}{|x-y|} \bigg| < 2\epsilon.
$$
If $V \geq 0$, then this implies
$$
\sup_{|y|>\max(R_0, R_1)} \int_{\R^3} \frac {\chi_{|x|>\tilde R} V(x) \dd x}{|x-y|} < 2\epsilon
$$
for any $\tilde R \geq R$. By dominated convergence, for $\tilde R \geq 2\max(R_0, R_1)$
$$
\sup_{|y| \leq \max(R_0, R_1)} \int_{\R^3} \frac {\chi_{|x|>\tilde R} V(x) \dd x}{|x-y|} \leq 2 \int_{\R^3} \frac {\chi_{|x|>\tilde R} V(x) \dd x}{|x|} \to 0 \leq 2 \|V\|_\K
$$
as $\tilde R \to \infty$. Hence $\|\chi_{|x| \geq R} V\|_\K \to 0$, in other words $V$ satisfies the modified distal condition, and $V$ is well-approximated in $\K \subset \Delta L^\infty$ by measures of compact support.

If $\Delta^{-1} V \in C_0$, then $\Delta^{-1} [\epsilon^{-3} \chi(\cdot/\epsilon) \ast V] \in C_0$ as well, so we can use the above to further approximate these mollifications in $\K$ by smooth cutoffs that are test functions.
\end{proof}

\begin{lemma}\lb{localdistal} For $V \geq 0 \in \K$, $V \in \Delta C_0$, meaning that
$$f_V(y):=\int_{\R^3} \frac{V(x) \dd x}{|x-y|} = 4\pi (-\Delta)^{-1} V \in C_0,
$$
if and only if $V$ has the local and modified distal Kato properties.
%
%
%
\end{lemma}

When $V \geq 0 \in \K$, $f_V$ is bounded; also, by Fatou's lemma, $f_V$ is lower semicontinuous, so it reaches a minimum on any compact set.


The local and distal conditions can be separated: if $V \geq 0$ only satisfies the local Kato condition, then it can be approximated in $\Delta L^\infty$ by smooth functions that need not decay at infinity, hence $f_V \in C$. The modified distal condition by itself states that $V$ decays at infinity in the Kato norm sense.

The local and distal Kato conditions are stated for $|V|$, so when $V$ has arbitrary sign or is even complex-valued the two conditions are equivalent to $V_\pm$ or $\re V_\pm$ and $\im V_\pm$ being in $\Delta C_0$.

The local and modified distal Kato conditions define closed subspaces of $\K$ that include $\K_0$, since they hold for test functions and are stable under taking limits in $\K$. However, these closed subspaces are strictly wider, since they also contain singular measures.

Indeed, if $V \in \D$, then $f_V$ is smooth and decays at infinity like $|x|^{-1}$. Any $V \in \K_0$ (and then $|V|$ as well) is well approximated in $\K$ by test functions, which gives a uniform approximation of $f_V$ by $C_0$ functions, so $f_V \in C_0$ itself.

\begin{proof}
%

If $V$ has the local Kato property, the function $f_V$ is uniformly approximated by
$$
f_{V, \delta}(y) = \int_{|x-y|>\delta} \frac {V(x) \dd x}{|x-y|}.
$$
The derivative of $f_{V, \delta}$ is of size
$$
|\nabla f_{V, \delta}(y)| \leq \int_{|x-y|>\delta} \frac {V(x) \dd x}{|x-y|^2} + \int_{|x-y|=\delta} \frac {V(x) \dd x}{|x-y|} \leq \delta^{-1} \|V\|_\K + \delta^{-1} \int_{|x-y|=\delta} V(x) \dd x.
$$
The second term may be infinite, but integrating over the line segment from $y_1$ to $y_2$ with $|y_1-y_2|<\delta/2$ leads to
$$
|f_{V, \delta}(y_1)-f_{V, \delta}(y_2)| \leq |y_1-y_2| \delta^{-1}\|V\|_\K + \delta^{-1} \int_{\delta/2 \leq |x-y| \leq 3\delta/2} V(x) \dd x \les |y_1-y_2| \delta^{-1}\|V\|_\K + \int_{|x-y| \leq 3\delta/2} \frac {V(x) \dd x}{|x-y|}.
$$
Supposing that $\sup_y \int_{|x-y| \leq 3\delta/2} \frac {V(x) \dd x}{|x-y|} < \epsilon$, it follows that for sufficiently small $|y_1-y_2|$
$$
|f_{V, \delta}(y_1)-f_{V, \delta}(y_2)| < 2\epsilon,\ |f_V(y_1)-f_V(y_2)| < 4\epsilon.
$$


Consequently $f_V$ is uniformly continuous, meaning that $V \in \Delta UC$ is well-approximated in $\Delta L^\infty$ by its (smooth) mollifications $V \ast \epsilon^{-3} \chi(\cdot/\epsilon)$.

If $V$ also fulfills the modified distal condition, we can first approximate it in $\K$ by a cutoff $\chi_{|x|>R} V(x)$, then do the above for the cutoff, resulting in an approximation of $V$ by test functions, so $f_V \in C_0$ and $V \in \Delta C_0$.

Conversely, assume that $V \geq 0 \in \K$. If $f_V$ vanishes at infinity, then, repeating the proof of the previous lemma, for every $\epsilon>0$ there exists $R_0$ such that $f_V(y)<\epsilon$ for $|y|>R_0$.

On the other hand, for $|y| \leq R_0$ and $|x|>R \geq 2R_0$,
$$
\int_{|x|>R} \frac {V(x) \dd x}{|x-y|} \leq 2 \int_{|x|>R} \frac {V(x) \dd x}{|x|} \to 0
$$
as $R \to \infty$ by dominated convergence. It follows that for large $R$ $\|\chi_{|x| > R} V\|_\K\| < \epsilon$, proving that $V$ has the modified distal property.

If in addition $f_V$ is continuous, we proceed by contradiction: assume that $V \in \Delta C_0$ does not have the local Kato property. Then there exist $\epsilon_0>0$ and a sequence $(y_n)_n$ such that
$$
\int_{|x-y_n|<1/n} \frac {V(x) \dd x}{|x-y_n|} \geq \epsilon_0.
$$
In light of the modified distal Kato property, all these points are inside a fixed ball, so a subsequence has some limit point $y_0$. Let $\delta>0$ and
$$
f_{V(y_0, \delta)}(y) = \int_{|x-y_0|>\delta} \frac {V(x) \dd x}{|x-y|}
$$
Since $f_{V(y_0, \delta)}(y)$ is continuous on $B(y_0, \delta/2)$, under our assumption that $f_V$ itself is continuous it follows that
$$
\int_{B(y_0, \delta)} \frac {|V(x)| \dd x}{|x-y|}
$$
is also continuous on $B(y_0, \delta/2)$.

For sufficiently large $n$, $B(y_n, 1/n) \subset B(y_0, \delta)$, so
$$
\int_{B(y_0, \delta)} \frac {|V(x)| \dd x}{|x-y_n|} \geq \epsilon_0.
$$
Passing to the limit in $n$ gives us that
$$
\int_{B(y_0, \delta)} \frac {|V(x)| \dd x}{|x-y_0|} \geq \epsilon_0
$$
for each $\delta>0$. However, this is a contradiction, as this integral vanishes when $\delta \to 0$ by dominated convergence (because $V \in \K$ cannot contain a point mass).
\end{proof}

So far we have looked at four closed subspaces of $\K$, determined by the local, distal, and modified distal Kato properties and by local integrability. It turns out that $\K_0$ is also characterized by a combination of these properties:
\begin{lemma}\lb{charkato} For a measure $V \in \K$, $V \in \K_0$ if and only if $V$ has the local and modified distal Kato properties and $V \in L^1_{loc}$.
\end{lemma}
\begin{proof} If $V \geq 0 \in C_c$, then $f_V$ is continuous and decays at infinity like $|x|^{-1}$, so $V \in \Delta C_0$. Any $V \in \K_0$ is approximated in $\K$ by $V_k \in \D$, so $|V|$ is well approximated in $\K$ by $|V_k| \in C_c$; thus $|V| \in \Delta C_0$. Then $|V|$, hence $V$ itself, has the local and modified distal Kato properties.

Likewise, $\D \subset L^1_{loc} \cap \K$ and $L^1_{loc} \cap \K$ is a closed subspace of $\K$ due to (\ref{frostman}).

Conversely, let $\chi$ be a standard cutoff function. Suppose that $V \in L^1_{loc}$ has the modified distal Kato property; then $V$ is well-approximated in $\K$ by its cutoffs $\chi(\cdot/R) V \in L^1 \cap \K$.

Without loss of generality, then, consider compactly supported $V \in L^1 \cap \K$. For $|\delta|<\epsilon$
$$
f_{|V(\cdot-\delta)-V|}(y) = \int_{\R^3} \frac {|V(x-\delta)-V(x)|}{|x-y|} \leq \epsilon^{-1} \|V(\cdot-\delta)-V\|_{L^1} + 2 \int_{B(0, 2\epsilon)} \frac {|V(x)|}{|x-y|}.
$$
If $V$ has the local Kato property, the second quantity goes uniformly to $0$ as $\epsilon \to 0$. By making $\delta$ accordingly small, the whole expression vanishes. Therefore $\|V(\cdot-\delta)-V\|_\K \to 0$ or in other words $V$ is continuous under translation in $\K$. Consequently, $V$ is well-approximated in $\K$ by its mollifications $V \ast \epsilon^{-3} \chi(\cdot/\epsilon) \in \D$, finishing the proof.
\end{proof}

\begin{observation}
Consider $g \in L^\infty$. If $V \in \K \cap \Delta C_0$, $V$ has the local and (modified) distal properties, hence so does $Vg$, implying that $Vg \in \Delta C_0$ as well. Thus, $\K \cap \Delta C_0$ is a closed $L^\infty$-submodule of $\K$ and same for $\K \cap \Delta L^\infty_0$ and $\K \cap L^1_{loc}$. The set of $V \in \K$ having the local Kato property is obviously another and same for the distal Kato property.
\end{observation}

\subsection{Properties of the Kato--Birman operator}
With the notation $\lambda=\eta^2$, let the free resolvent be
$$
(-\Delta-\lambda)^{-1}(x, y)=R_0(\eta^2)(x, y) = \frac 1 {4\pi} \frac {e^{i\eta|x-y|}}{|x-y|}.
$$
We see that the free resolvent has holomorphic extensions to the Riemann surface of $\sqrt \lambda$, which is a two-sheeted cover of $\C$ with a branching point at $0$. Then $\C \setminus [0, \infty)$ is just one of the sheets, corresponding to the upper half-plane $\im \eta > 0$, whose boundary $\lambda\pm i0 \in [0, \infty)$ corresponds to $\eta \in \R$.

Almost the entire argument uses this sheet and its boundary. When $V \in \K$, operators on the other sheet are only bounded in exponentially weighted spaces with positive weights. However, the proof of Proposition \ref{no_pos} uses a strip in the other sheet.

For $V \in \K$, consider the holomorphic family of operators on $\C \setminus [0, \infty)$ $A(\lambda)=V R_0(\lambda)$. When $V \in \K_0$, both $A$ and $A^*$ enjoy several continuity and compactness properties, up to the boundary. These properties still hold under weaker assumptions on $V$, see below.

On the usual sheet $\im \eta \geq 0$, including its boundary, $A(\eta^2)$ is uniformly bounded in $\B(\M) \cap \B(\K)$ and $\partial_\eta A(\eta^2)$ is uniformly bounded in $\B(\M, \K)$. If $V \in \K$ fulfills the distal condition then
$$
A(\eta^2) = A(\eta_0^2) + (\eta-\eta_0) \partial_\lambda A(\eta_0^2) + o_{\B(\M, \K)}(\eta-\eta_0)
$$
uniformly over $\im \eta, \im \eta_0 \geq 0$. This is because
$$
\frac {e^{i\eta|x-y|}-e^{i\eta_0|x-y|}}{(\lambda-\lambda_0)|x-y|} - 1 \les \min(1, (\eta-\eta_0)|x-y|).
$$
Higher-order complex derivatives live in weighted spaces on the boundary (on the spectrum).

We next prove certain properties of $A(\lambda)=R_0(\lambda) V$ when $V \in \K$ has the local and modified distal Kato properties. The Kato--Birman operator $I+A$ plays an important part in the spectral analysis of $H=-\Delta+V$.

The operator family $A(\lambda)=V R_0(\lambda)$ is holomorphic on $\C \setminus [0, \infty)$, in the sense that $\partial_\lambda A \in \B(\M) \cap \B(\K)$.

If $V \in \K$, then $V$ is a relatively bounded perturbation of $-\Delta$, in the sense that
$$
\sup_{\lambda \in \C} \|V R_0(\lambda)\|_{\B(\M)} \leq \frac {\|V\|_\K} {4\pi},
$$
including the boundary values of $R_0$ along $[0, \infty)$. If $V \in \K_0$, it is a relatively compact perturbation. If $V$ fulfills the distal condition, then $A(\lambda)$ is continuous in the $\B(\M)$ norm, see below. In addition to $\B(\M)$, the same boundedness, compactness, and continuity properties hold in $\B(\K)$.

Likewise, for the adjoint, if $V \in \K$ then $R_0(\lambda) V \in \B(L^\infty) \cap \B(\K^*)$ uniformly in $\lambda$, where by $L^\infty$ we understand the Banach space of bounded Borel measurable functions, and the operators $R_0(\lambda) V$ are compact in $B(L^\infty, C_0)$ and $\B(\K^*, (\K^*)_0)$ if $V \in \K_0$.

The similar behaviors of $A$ and $A^*$ are dictated by the following commutation relation: $V A^*(\lambda) = A(\ov \lambda) V$.

Consider an underlying measure $\mu$ such that $V \in L^1_{loc}(\mu)$. By complex interpolation, after fixing an underlying measure $\mu$ such that $V \in L^1_{loc}(\mu)$, we obtain similar bounds for $|V|^\alpha \sgn V R_0 |V|^{1-\alpha} \in \B(L^p)$ and $|V|^\alpha R_0 |V|^{1-\alpha} \sgn V \in \B(L^p)$ for $\alpha \in (0, 1)$ and $p=1/\alpha$.

In particular,
$$
|V|^{1/2} R_0 |V|^{1/2},\ |V|^{1/2} R_0 |V|^{1/2} \in \B(L^2)
$$
uniformly and are compact when $V \in \K_0$.

If $V$ is not locally integrable, a more general intrinsic approach (without invoking an underlying measure) is as follows. Let $\mu=|V|$ and consider the operator $R_0(\lambda)V$ on $L^1_{loc}(\mu)$. Then $V \in \K$ means that $\frac 1 {|x-y|} \in L^\infty_y L^1_x(\mu)$, so $[R_0(\lambda) V](x, y) \in \B(L^1_\mu)$. As already mentioned, $R_0(\lambda) V \in \B(L^\infty)$, so by interpolation $R_0(\lambda) V \in \B(L^p_\mu)$ for $1 \leq p \leq \infty$.

\begin{lemma} $\M \cap \K \subset [\M, \K]_{1/2} \subset \dot H^{-1}$.
\end{lemma}

\begin{proof}
Consider the quadratic form $\langle V_1, (-\Delta)^{-1} V_2 \rangle$. It is bounded on $\M \times \K$ and $\K \times \M$, so, by interpolation, it is bounded on $[\M, \K]_{1/2} \times [\M, \K]_{1/2}$. Hence
$$
\|V\|_{\dot H^{-1}}^2 = \langle V, (-\Delta)^{-1} V \rangle \les \|V\|_{[\M, \K]_{1/2}}^2
$$
and the conclusion follows.
\end{proof}

\begin{lemma} The statement that $V \in \B(\dot H^1, \dot H^{-1})$ is true for either $\dot H^1=(-\Delta)^{-1/2} L^2$ or $\dot H^1(\mu)=(-\Delta)^{-1/2} L^2(\mu)$ and for either $\dot H^{-1}$ or $\dot H^{-1}(\mu)=(-\Delta)^{1/2} L^2(\mu)$.
\end{lemma}
\begin{proof}
Recall that $\K \subset \Delta L^\infty$ and likewise $\Delta^{-1} \M \subset \K^*$. Thus $V \in \K$ acts as a bounded multiplier from $\Delta^{-1} L^1$ or $\Delta^{-1} L^1(\mu)$ (in general from $\Delta^{-1} \M$) to $\M$ and more specifically to $L^1(\mu)$, as well as from $L^\infty$ to $\K \subset \Delta L^\infty$; hence, by complex interpolation, also from $\dot H^1$ to $[\M, \K]_{1/2} \subset \dot H^{-1}$.

Then we can obtain all four possibilities by complex interpolation:
$$
[\Delta^{-1} L^1, L^\infty]_{1/2}=\dot H^1,\ [\Delta^{-1} L^1(\mu),\ L^\infty]_{1/2}=\dot H^1(\mu),
$$
$$
[L^1(\mu), \Delta L^\infty]_{1/2}=\dot H^{-1}(\mu),\ [\M, \K]_{1/2} \subset \dot H^{-1}.
$$
\end{proof}

%


We can generalize Lemma \ref{localdistal} in two ways. One is that if $V \in \K$ has the local and modified distal Kato properties, the same is true for $Vg$, where $\|g\|_{L^\infty} \leq 1$, with the same moduli of continuity for the two properties. Hence $f_{Vg}$ form a compact family in $C_0$.

The other is that we can replace $(-\Delta)^{-1}$ by $R_0(\eta^2)$, losing the equivalence, but preserving one implication. This results in
\begin{lemma} If $V \in K$ has the local and modified distal Kato properties, $R_0(\eta^2) V \in \B(L^\infty, C_0)$ is approximated in norm by translations, hence also by mollifications $\epsilon^{-3} \chi(\cdot/\epsilon) \ast R_0(\eta^2) V$.
\end{lemma}

We use this in the proof of the next statement.

\begin{proposition} If $V \in \K$ has the local and modified distal Kato properties, then for $\im \eta \geq 0$ the families $A(\eta^2)$ and $A^*(\eta^2)$ are continuous up to the boundary in the $\B(L^\infty, C_0)$ and $\B(\K^*)$, respectively $\B(\M)$ and $\B(\K)$ norms, and compact.
\end{proposition}

\begin{proof} If $V \in \K$, then $R_0(\eta^2)(x, y) V(y) \in L^\infty_x \M_y \subset \B(L^\infty)$, where we again consider only Borel measures and Borel measurable functions.

If $V$ has the distal Kato property, then $R_0(\eta^2) V$ is uniformly approximated in $L^\infty_x \M_y$ by $\chi_{|x-y|<R} R_0(\eta^2)(x, y) V(y)$. However,
$$
\|\chi_{|x-y|<R} \partial_\eta R_0(\eta^2)(x, y) V(y)\|_{L^\infty_x M_y} \leq \frac {R}{4\pi} \|V\|_\K,
$$
so $R_0(\eta^2) V \in \B(L^\infty)$ and $VR_0(\eta^2) \in \B(\M)$ are continuous in $\eta$, as the uniform limit of this family of continuous approximations.

By the same method as in the proof of Lemma \ref{localdistal},
for each fixed $\eta$, $R_0(\eta^2) V \in UC$. Its modulus of continuity only depends on $\eta$, $\|V\|_\K$, and the uniform vanishing rate from the local Kato property.

All these properties remain the same when replacing $V$ by $Vg$, where $\|g\|_{L^\infty} \leq 1$, so the image of the unit ball $\{R_0(\eta^2) V g \mid \|g\|_{L^\infty} \leq 1\}$ is uniformly equicontinuous, for bounded $\eta$.

If $V$ has the modified distal Kato property, $R_0(\eta^2) Vg$ also vanish uniformly (including in $\eta$) at infinity. Hence, by the Arzela-Ascoli theorem, for each $\eta$ $R_0(\eta^2) V \in \B(L^\infty, C_0)$ is compact.

A dual point of view is that, when $V$ has the local Kato property, $\frac {V(x)}{|x-y|} \in \M_x$ form a continuous family of measures and when $V$ has the modified distal Kato property this family vanishes at infinity. The same is true for $V(x)R_0(\eta^2)(x, y)$ for fixed $\eta$.

Integration against a $C_0$ (measure-valued) function defines a compact operator on $\M$, because the closed convex shell of a compact set is compact. Hence $V R_0(\eta^2)$ is compact in $\B(\M)$.

Similarly, approximating $R_0(\eta^2) V$ by $\epsilon^{-3} \chi(\cdot/\epsilon) \ast R_0(\eta^2) V$, $[\epsilon^{-3} \chi(\cdot/\epsilon) \ast R_0(\eta^2)](x, y)$ form a continuous family in $\K^*_x$, parametrized by $y$. If $V$ has the modified distal property, we can approximate $V$ by a bounded support cutoff $\chi_{|y| \leq R} V(y)$, such that for $y$ in this compact set $[\epsilon^{-3} \chi(\cdot/\epsilon) \ast R_0(\eta^2)](x, y)$ form a compact set in $\K^*_x$.

For $g \in \K^*$ $\|Vg\|_\M \leq \|V\|_\K \|g\|_{\K^*}$, so, when $\|g\|_{\K^*} \leq 1$, $\epsilon^{-3} \chi(\cdot/\epsilon) \ast R_0(\eta^2) Vg$ belongs to the closed convex shell of a compact set in $\K^*$, which is itself compact. Hence $R_0(\eta^2) V \in \B(\K^*)$ is approximated in norm by compact operators and is itself compact.

This implies that the adjoint $V R_0(\eta^2) \in \B(\K)$ is also compact.
\end{proof}

We next consider the two conditions for Wiener's theorem. Let
$$
S_0(t)(x, y)=\chi_{t \geq 0} \frac 1 {4\pi t} \delta_t(|x-y|)
$$
and
$$
S(t)(x, y)=V(x)S_0(t)(x, y).
$$
Recall the conditions for Wiener's theorem:
$$
\lim_{R \to \infty} \|\chi_{\rho \geq R} S(\rho)\|_{L^\infty_y \M_x L^1_\rho} = 0 \text{ (decay at infinity)}
$$
and for some $N \geq 1$
$$
\lim_{\epsilon \to 0} \|S^N(\rho+\epsilon)-S^N(\rho)\|_{L^\infty_y \M_x L^1_\rho} = 0 \text{ (continuity under translation)}.
$$

Clearly, the distal Kato condition is equivalent to the decay at infinity condition. The second, though, has a different nature.
\begin{lemma}\lb{second_cond} If $V \in \K_0$, meaning that $V \in \K$ has the local and modified distal Kato properties and $V \in L^1_{loc}$, then $S^2(t)$ is continuous under translation.
\end{lemma}
The case when $V \not \in L^1_{loc}$ (is a singular measure) will be treated elsewhere.
\begin{proof}
If $V$ has the local and distal Kato properties, $S_0(t)V$ is well-approximated in $L^\infty_x \M_{y, t}$ by $\chi_{\epsilon \leq t \leq R} S_0(t) V$ and, if $V$ also has the modified distal Kato property, by $\chi_{\epsilon \leq t \leq R} S_0(t)$ $\chi_{|y| \leq R_0} V$.

Hence with no loss of generality we can take $V \geq 0 \in \M \cap \K$ of compact support and prove continuity under translation simply for (a power of) $\chi_{\epsilon \leq t \leq R} S_0(t) V$. It then further suffices to prove it for (a finite product of) small neighborhoods of every $t>0$.

For $[a_1, b_1], [a_2, b_2] \subset (0, \infty)$, consider the product
$$
S_1=\chi_{[a_1, b_1]}(t) S_0(t) V \ast_t \chi_{[a_2, b_2]}(t) S_0(t).
$$
For each $x,y$ and for $t \in [a_1+a_2, b_1+b_2]$, its value is given by the integral
$$
S_1(x, y)(t) = \int_{\substack{|x-z|+|z-y|=t \\ a_1 \leq |x-z| \leq b_1 \\ a_2 \leq |z-y| \leq b_2}} \frac {V(z) \dd z}{|x-z| |z-y|},
$$
on a piece of the revolution ellipsoid $\{z \mid |x-z|+|z-y|=t\} \subset \R^3$.

If $V \in L^1_{loc}$, then its compact support restriction $\chi_{|y| \leq R_0} V$ is in $L^1$. By Fubini's theorem, $S_1 \in L^\infty_{x, y} L^1_t$ and
$$
\|S_1\|_{L^\infty_{x, y} L^1_t} \les \|V\|_{L^1},
$$
with a factor that depends on the curvature of the ellipsoids. However, due to our previous reductions, the curvature is bounded from above (and below, but not important here) on the pieces we consider.

Approximating $V$ in the $L^1$ norm by continuous functions, we see that in fact $S_1 \in C_{x, y} L^1_t$. Then, for $x$ and $y$ in a compact set, the functions $S_1(x, y)(t)$ form a compact family in $L^1_t$. By the Kolmogorov-Fr\'{e}chet theorem, they are equicontinuous under translation, which implies the continuity under translation of $S_1$ in $L^\infty_{x, y} L^1_t$.

Hence the corresponding piece of $S^2=S_0V \ast S_0V$ is continuous under translation in $\K_y L^\infty_x L^1_t$ (which is stronger than needed, but true only for the approximation on a finite $t$ interval, not at $0$ or infinity) and the conclusion follows.
\end{proof}

We end this section with a characterization of the dual space $\K^*$.

\begin{proposition}\lb{katodual} $g \in \K^*$ if and only if there exists a positive measure $\mu$, with $\|\mu\|_\M=\mu(\R^3)=\|g\|_{K^*}$, such that $|g| \leq (-\Delta)^{-1} \mu$ almost everywhere.
\end{proposition}
As an obvious consequence, which can be proved independently, bounded functions of compact support are in $\K^*$. However, $\K^*$ is not separable, just as $L^{3, \infty}$ and in general $L^{p, \infty}$ are not separable.

One can also consider the closure of the class of test functions $(\K^*)_0=\ov {\D}_{\K^*} \subset L^{3, \infty}_0$, which is separable.

\begin{proof} One direction is almost obvious: if such a measure $\mu$ exists, then $g \in \K^*$ with $\|g\|_{\K^*} \leq \|\mu\|_\M$.

Conversely, given $g \in \K^*$, constructing a corresponding measure $\mu$ is more difficult.

With no loss of generality take $g \geq 0$; it also suffices to test it against positive functions. To begin with, consider the set
$$
M=\{\mu \in \M \mid \mu \geq 0, (-\Delta)^{-1} \mu \leq (-\Delta)^{-1} \delta_0\}.
$$
The set $M$ is nonempty because it contains $\delta_0$, but it also contains $\frac 1 {|B(0, R)|} \chi_{B(0, R)}$ for $R>0$, where $|B(0, R)|=4\pi R^3/3$, and many other measures.

It is easier to only consider the absolutely continuous measures in $M$ (and, in fact, one can consider only a certain closure of the spherical averages mentioned above). If $\mu \in M$ then $\mu(-x) \in M$, so this sign makes no difference.

Let
$$
g^\#(x) = \sup_{\mu \in M} (\mu*g)(x).
$$
Since $g \in L^{3/2, 1}$, for almost every $x$ $\lim_{R \to 0} \frac 1 {|B(0, R)|} (\chi_{B(0, R)}*g)(x)=g(x)$, so $g^\#(x) \geq g(x)$ almost everywhere. Also note that $((-\Delta)^{-1} \delta_x)^\# = (-\Delta)^{-1} \delta_x$.

Let
$$
\tilde g(x)= \int_{\R^3} \mu_x(x-y) g(y) \dd y
$$
for some measurable choice of $\mu_x \in M$. Then for $f \geq 0$
$$
\int_{\R^3} f(x) \tilde g(x) \dd x = \int_{\R^3} \breve f(x) g(x) \dd x,
$$
where
$$
\breve f(y) = \int_{\R_3} \mu_x(x-y) f(x) \dd x.
$$
The two operators are different: while $\tilde g$ comes from a maximal operator, $\breve f$ is related to stopping time arguments.

Since $g \in \K^*$,
$$
\int_{\R^3} \breve f(x) g(x) \dd x \leq \|\breve f\|_\K \|g\|_{\K^*}.
$$
But
$$
\|\breve f\|_\K = 4\pi \|(-\Delta)^{-1} \breve f\|_{L^\infty} = \sup_x \int_{\R^3} \frac 1 {|x-y|} \int_{\R^3} \mu_z(z-y) f(z) \dd z \dd y.
$$
Interchanging the order of integration, since $\mu_z \in M$
$$
\int_{\R^3} \frac 1 {|x-y|} \mu_z(z-y) \dd y \leq \frac 1 {|x-z|}.
$$
Therefore $\|\breve f\|_\K \leq \|f\|_\K$. It also follows that $\|\tilde g\|_{\K^*} \leq \|g\|_{\K^*}$, then that $\|g^\#\|_{\K^*} \leq \|g\|_{\K^*}$, and finally that they are equal, since $g^\# \geq g$. In particular, this shows that $g^\# \in L^{3, \infty}$.

Next, we prove that $g^\#$ is superharmonic. It suffices to show that for each $x$ and $R>0$
$$
g^\#(x) \geq \frac 1 {|B(x, R)|} \int_{|x-y| \leq R} g^\#(y) \dd y,
$$
which further reduces to proving that for any measurable choice of $\mu_x \in M$
$$
g^\#(x) \geq \frac 1 {|B(x, R)|} \int_{|x-y| \leq R} \int_{\R^3} \mu_y(y-z) g(z) \dd z \dd y.
$$
Interchanging the order of integration, we need to bound the expression
$$
I_x(z) = \frac 1 {|B(x, R)|} \int_{|x-y| \leq R} \int_{\R^3} \mu_y(y-z) \dd y.
$$
This is not translation-invariant, but the process is the same for each $x$. Note that
$$\begin{aligned}
(-\Delta)^{-1} I_x &= \frac 1 {4\pi |B(x, R)|} \int_{\R^3} \frac 1 {|z-z_0|} \int_{|x-y| \leq R} \int_{\R^3} \mu_y(y-z) \dd y \dd z \\
&\leq \frac 1 {4\pi |B(x, R)|} \int_{|x-y| \leq R} \frac 1 {|y-z_0|} \leq (-\Delta)^{-1} \delta_x,
\end{aligned}$$
since $\mu_y \in M$ and $\frac 1 {|B(0, R)|} \chi_{B(0, R)} \in M$ as well. But then
$$
g^\#(x) \geq \int_{\R^3} I_x(z) g(z) \dd z
$$
by the definition of $g^\#$, finishing the proof of superharmonicity.

Since $g^\#$ is superharmonic and decays at infinity, there exists a positive measure $\mu \geq 0$ such that $g^\#=(-\Delta)^{-1} \mu$. Then for $f \in \K$, $f \geq 0$, and $h=(-\Delta)^{-1} f \geq 0$, bounded and superharmonic,
$$
\int_{\R^3} f g^\# \dd x= \int_{\R^3} h \mu \leq \|h\|_{L^\infty} \|g\|_{\K^*}.
$$

Let $h$ be such that $0 \leq h \leq 1$ and $h=1$ on $B(0, R)$ (this can be obtained for example as $(-\Delta)^{-1}$ applied to a uniform measure on the boundary of a ball). Then $\mu(B(0, R)) \leq \|g\|_{\K^*}$. By passing to the limit we obtain that $\mu \in \M$ and $\|\mu\|_\M \leq \|g\|_{\K^*}$, which we wanted to prove.
\end{proof}

\section{Spectrum of the Hamiltonian}\lb{b}

\subsection{Preliminaries} Assuming that $V$ has the local and modified distal Kato properties, again note that $A(\lambda)=V R_0(\lambda) \in \B(\M)$ is compact for all $\lambda \in \C \setminus [0, \infty)$. By Weyl's theorem \cite{resi4} \cite{weyl} it follows that, when $V \in \K_0$, $-\Delta$ and $H$ have the same essential spectrum, namely $\sigma_{ess}(H)=\sigma_{ess}(-\Delta)=[0, \infty)$.

However, much more is true: $A(\lambda)$ is not only holomorphic and compact inside $\C \setminus [0, \infty)$, but also continuous (hence uniformly bounded) and compact up to the boundary, with continuous and bounded derivatives (all except the first in weighted spaces) on the boundary.

On the boundary $[0, \infty)$, the complex derivative of $V R_0(\lambda)$ is still continuous and bounded, but only in $\B(\M, \K)$. The second and higher complex derivatives live in weighted spaces with progressively higher weights: $\partial^k_\lambda A(\lambda) \in \B(\langle x \rangle^{1-k} \M, \langle x \rangle^{k-1} \K)$.

Cauchy's theorem and Cauchy's formulae are true due to the continuity of $A(\eta^2)$ up to the boundary. Including the other sheet of the Riemann surface, similar results hold in exponentially weighted spaces.

Furthermore, the same holds uniformly when replacing $V$ by $tV$ for $|t| \leq 1$. This leads to a much more precise characterization of $\sigma(H)$.

To determine the discrete spectrum of $H$, we shall use the following version of Rouch\'e's theorem, see \cite[Theorem 2.2]{gohsig}. We state a simplified version, assuming analiticity (no poles):
\begin{theorem}[Rouch\'e/Gohberg--Sigal]\lb{rouche} Let $O \subset \C$ be a simply connected open region with rectifiable boundary $\partial O$ and let $A:\ov O \to \B(X, \tilde X)$ be a family of operators between two Banach spaces, analytic and Fredholm ($\dim \ker A(z)<\infty$, $A(z)(X)$ is closed in $\tilde X$, and $\dim \tilde X/A(z)(X)< \infty$) at each point $z \in O$ and continuous on $\ov O$, such that $A(z)$ is invertible for $z \in \partial O$. Consider another family of operators $B: \ov O \to \B(X, \tilde X)$, analytic in $O$ and continuous on $\ov O$, such that $\sup_{z \in \partial O} \|A^{-1}(z)B(z)\|_{\B(X)} <1$. Then $A$ and $A+B$ have the same number of zeros in $O$.
\end{theorem}

To establish invertibility on the boundary, we use the following extremely useful result of Fechbach \cite{yaj}:
\begin{lemma}\lb{fechbach} Let $X=X_0 \oplus X_1$ and $Y=Y_0 \oplus Y_1$ be direct sum decompositions of vector spaces $X$ and $Y$. For $L \in \B(X, Y)$ consider its matrix decomposition
$$
L=\begin{pmatrix}
L_{00} & L_{01}\\
L_{10} & L_{11}
\end{pmatrix},
$$
where $L_{ij} \in \B(X_j, Y_i)$. Assume $L_{00}$ is invertible. Then $L$ is invertible if and only if $D=L_{11}-L_{10} L_{00}^{-1} L_{01}$ is invertible, in which case
$$
L^{-1}=\begin{pmatrix}
L_{00}^{-1} + L_{00}^{-1} L_{01} D^{-1} L_{10} L_{00}^{-1} & -L_{00}^{-1} L_{01} D^{-1}\\
-D^{-1} L_{10} L_{00}^{-1} & D^{-1}
\end{pmatrix}.
$$
\end{lemma}

\begin{observation}
Often it suffices to consider $X=Y$ and $X_i=Y_i$.

Suppose that $D f = 0$. Then $\ds L \begin{pmatrix} -L_{00}^{-1} L_{01} f \\ f \end{pmatrix} = 0$.
\end{observation}

While this is suitable in the context of Hilbert spaces, where it is easy to obtain direct sum decompositions, in the more general case of Banach spaces one deals with short exact sequences, which need not split. Below we prove a more general result that holds in this generality.

If $X$ is a normed vector space, a natural norm on any quotient space $X/X_0$ is $\|x+X_0\|_{X/X_0}=\inf_{x_0 \in X_0} \|x+x_0\|_X$. In general, this is a seminorm and only becomes a norm if $X_0 \subset X$ is closed. The projection $P:X \to X/X_0$, $P(x)=x+X_0$, is bounded with a norm of at most $1$: $\|Px\|_{X/X_0} \leq \|x\|_X$ (and exactly $1$ in Banach spaces by Riesz's lemma).

Another situation in which a canonical norm is defined is when $X_0=\ker L$ for some operator $L$ bounded on $X$. Then $X/\ker L \simeq L(X)=\im L$ and a natural norm on $L(X)$ is $\|y\|_{L(X)}=\inf_{Lx=y} \|x\|_X$. The induced map $\ov L: X/X_0 \to L(X)$ is an isometry.

If $L(X) \subset Y$, then another norm on $L(X)$ is the restriction of the $Y$ norm. The two norms are not necessarily equivalent: for $y \in L(X)$
$$
\|y\|_X \leq \|L\|_{\B(X)} \inf_{Lx=y} \|x\|_X = \|L\|_{\B(X)} \|y\|_{L(X)},
$$
but the converse need not be true. However, the restriction of the $Y$ norm to $L(X)$ induces a new norm on $X/X_0$, namely $\|x+X_0\|=\|Lx\|_X$, such that $\ov L$ is again an isometry.

Below is the more general statement that we use instead of Fechbach's lemma.

\begin{lemma}\lb{fech} Consider a short exact sequence of vector spaces $0 \to X_0 \to X \to X_1 \to 0$, meaning $X/X_0 \simeq X_1$, and a linear operator $L:X \to L(X) \subset Y$.

Consider the restriction $L_0=L \mid_{X_0}:X_0 \to L(X_0) \subset Y$ and the induced map $L_1: X_1=X/X_0 \to L(X)/L(X_0)$, given by $L_1(x+X_0)=Lx+L(X_0)$. Then $L$ is injective if and only if $L_0$ and $L_1$ are injective, i.e.~if and only if $\ker L \cap X_0=\{0\}$ and if $Lx \in L(X_0)$ implies $x \in X_0$.

Consider the projection $P:L(X) \to L(X)/L(X_0)$.  If $X$ and $Y$ are normed spaces and $L$, $L_0^{-1}$, and $L_1^{-1}$ are bounded, then
$$
\|L^{-1}\|_{\B(L(X), X)} \leq \|L_0^{-1}\|_{\B(L(X_0), X_0)} (1+ \|L\|_{\B(X)} \|L_1^{-1} P\|_{\B(L(X), X/X_0)}).
$$
\end{lemma}

This lemma does not require the existence of a projection from $X$ onto $X_0$, which would let us split the short exact sequence into a direct sum. On $L(X)$ we use the $Y$ norm.

\begin{proof} If $L_0$ is not injective, then there exists $x_0 \in X_0$ such that $L(x_0)=L_0(x_0)=0$. Hence $L_0$ being injective is necessary for $L$ to be injective.

If $L_1$ is not injective, then there exists $x \in X$ such that $x \not \in X_0$ and $L(x) \in L(X_0)$. Hence $L_1$ being injective is necessary for $L$ to be injective.

Conversely, suppose that $L_0$ is injective. Take any $x \in X$ such that $Lx=0$. Then $L_1(x+X_0)=L(X_0)$. Since $L_1$ is injective, it follows that $x+X_0=X_0$, or in other words $x \in X_0$. Since $L_0$ is injective, it follows that $x=0$. Thus $L$ is also injective.

Furthermore, take any $y \in L(X)$. Since $L_1$ is surjective onto $L(X)+L(X_0)$, it follows that there exists $x \in X$ such that $Lx+L(X_0)=L_1(x+X_0)=Py=y+L(X_0)$. But then $y-Lx \in L(X_0)$, so there exists $x_0$ such that $y-Lx=Lx_0$, hence $y=L(x+x_0)$.

Finally, when $X$ is a normed space, $\|L^{-1}\|$ can be traced above as follows: for each $\epsilon>0$, the equivalence class $x+X_0=L_1^{-1}Py$ by definition has a representative $x$ such that
$$
\|x\|_X \leq \|L_1^{-1}P\|_{\B(L(X), X/X_0)} \|y\|_X + \epsilon;
$$
then
$$
\|y-Lx\|_{L(X_0)} \leq (1+\|L\|_{\B(X)} (\|L_1^{-1}P\|_{\B(L(X), X/X_0)}\|y\|_X + \epsilon).
$$
\end{proof}

If $\im L \subset Y$ is closed, then by Riesz's lemma the surjectivity of $L$ is equivalent to the injectivity of $L^*$, which can be characterized in a similar way: to the short exact sequence $0 \to X_0 \to X \to X_1 \to 0$ there corresponds the dual one $0 \to X_1^* \to X^* \to X_0^* \to 0$ and we can apply the same result to $T^*$.

We use this lemma when the short exact sequence comes from another operator, meaning there exists $\tilde L$ such that $X_0=\ker \tilde L$ and $X_1=\im \tilde L$. In fact, one can always represent a short exact sequence of vector spaces thusly, by taking $\tilde L$ to be the projection onto the quotient space. Then $\tilde L$ is bounded if and only if $X_0$ is closed.

Giving up on some generality, we henceforth assume that $\tilde L \in B(X, Y)$, so $X_0$ is closed. Furthermore, we shall apply this corollary when $L$ is a small perturbation of $\tilde L$, see below. This leads to a more explicit formulation.

\begin{corollary}\lb{bach} Let $X, Y$ be normed vector spaces and consider $L, \tilde L \in \B(X, Y)$. Let $X_0=\ker \tilde L$, $X_1=\im \tilde L$ and consider the restriction $L_0 = L \mid_{X_0}$ and the induced map $L_1:X/X_0 \to L(X)/L(X_0)$. Assume that there exist $C_0$ and $C_1$ such that
$$
\forall x_0 \in X_0\ \|Lx_0\|_Y \geq C_0 \|x_0\|_X
$$
and
$$
\forall x \in X \inf_{x_0 \in X_0} \|L(x+x_0)\|_Y \geq C_1 \inf_{x_0 \in X_0} \|x+x_0\|_X.
$$
Then
$$
\|L^{-1}\|_{\B(L(X), X)} \leq C_0^{-1}(1 + \|L\|_{\B(X)} C_1^{-1}).
$$
Suppose that the induced map $\ov {\tilde L}:X/\ker \tilde L \to \im \tilde L$ has a bounded inverse, meaning that there exists $C_2$ such that
$$
\forall x \in X\ \|\tilde L x\|_Y \geq C_2 \inf_{x_0 \in X_0} \|x+x_0\|_X.
$$
If $\|L-\tilde L\|_{\B(X, Y)} < \frac {C_2}{1+2C_0^{-1}\|L\|_{\B(X, Y)}}$, the second condition is then fulfilled with $C_1=C_2-\|L-\tilde L\|_{\B(X, Y)}(1+2C_0^{-1}\|L\|_{\B(X, Y)})$.
\end{corollary}

Under natural conditions, $\ov {\tilde L}$ has a bounded inverse:
\begin{lemma}
Let $L=J+K \in \B(X, Y)$, where $J$ is invertible and $K$ is compact. Then $\ov L:X/\ker L \to \im L$ has a bounded inverse.
\end{lemma}
\begin{proof} Replacing $L$ by $J^{-1} L$ and $K$ by $J^{-1} K$, we can assume $Y=X$ and $J=I$. Suppose the inverse of $\ov L$ is not bounded; then there exists a sequence $x_n \in X$ such that $\|L x_n\| \to 0$, but $\inf_{x_0 \in \ker L}\|x_n+x_0\| = 1$.

Since $K$ is compact, along some subsequence $K x_n \to K y$ for $y \in X$ with $\|y\| \leq 1$. But $x_n+Kx_n \to 0$ then implies $x_n \to -Ky$ (so $x_n$ converges). Then $K^2y=K(Ky)=-\lim_n Kx_n=-Ky$. Therefore $LKy=K^2y+Ky=0$, so $Ky \in \Ker L$.

As $x_n \to Ky$, this contradicts our assumption that $d(x_n, \Ker L)=1$ for each $n$.
\end{proof}

This leads to no concrete bounds on $C_2$, though. The first condition similarly reduces to $\ker L \cap \ker \tilde L = \{0\}$, which again does not lead to explicit bounds for $C_0$.

In the more general context of duality, the first condition is implied by
$$
\forall x_0 \in X_0\setminus\{0\}\ \exists y_0^* \in Y_0^*\setminus\{0\}\ |\langle Lx_0, y_0^* \rangle| \geq C_0 \|x_0\|_X \|y_0^*\|_{Y^*},
$$
where $Y_0^*=\ker \tilde L^*$. If the dual is also true, it shows that $L$ is bijective, by applying the same Corollary \ref{bach} to $L^*$. One can also use the predual instead of the dual.

In other words, for $L$, a small perturbation of $\tilde L$, to be bijective it suffices that the symplectic form defined by $L$ (or equivalently $L-\tilde L$) on $X_0 \oplus Y_0^*$ should be non-degenerate.

Likewise, the condition that $\ov {\tilde L}$ has a bounded inverse would be implied by the non-degeneracy of the symplectic form induced by $\tilde L$ on $(X/X_0) \oplus (Y^*/Y_0^*)$.

\begin{proof} For the most part, Corollary \ref{bach} is just a restatement of Lemma \ref{fech}.

Regarding the last statement, consider $x \in X$ and $x_1 \in X_0$ such that $\|x+x_1\|_X \leq 2 \inf_{x_0 \in X_0} \|x+x_1\|_X$.
Also take $x_2 \in X_0$ such that
$$
\|L(x+x_2)\|_Y \leq 2 \inf_{x_0 \in X_0} \|L(x+x_0)\|_Y \leq 2 \|L (x+x_1)\|_Y.
$$
Then
$$
\|L (x_2-x_1)\|_Y \leq \|L(x+x_2)\|_Y + \|L(x+x_1)\|_Y \leq 3 \|L (x+x_1)\|_X,
$$
so $\|x_2-x_1\|_X \leq 3 C_0^{-1} \|L\|_{\B(X, Y)} \|x+x_1\|_X$ and
$$
\|x+x_2\| \leq \|x+x_1\| + \|x_2-x_1\| \leq (1+3 C_0^{-1} \|L\|_{\B(X, Y)}) \|x+x_1\|_X.
$$
Hence
$$\begin{aligned}
\|L(x+x_2)\|_Y &\geq \|\tilde L x\| - \|(L-\tilde L)(x+x_2)\| \\
&\geq C_2 \inf_{x_0 \in X} \|x+x_0\|_X - \|L-\tilde L\| \|x+x_2\|_X \\
&\geq [\frac {C_2} 2 - \|L-\tilde L\|_{\B(X, Y)} (1 + 3 C_0^{-1} \|L\|_{\B(X, Y)})] \|x+x_1\|_X \\
&\geq [\frac {C_2} 2 - \|L-\tilde L\|_{\B(X, Y)} (1 + 3 C_0^{-1} \|L\|_{\B(X, Y)})] \inf_{x_0 \in X} \|x+x_0\|_X.
\end{aligned}$$
Consequently the same is true for $\inf_{x_0\in X_0} \|L(x+x_0)\|_Y$. By optimizing various coefficients (replacing $2$ by $1+\epsilon$) we obtain the desired conclusion.

This shows that, if $\ov {\tilde L}$ is invertible, so is the induced map $L_1$, when $L$ is a sufficiently small perturbation of $\tilde L$.
\end{proof}

\subsection{Main results}

We start determining the properties of $\sigma(H)$ with one of the simplest, which does not require Corollary \ref{bach}.

For convenience, let
\begin{definition} $\K_1=\{V \in \K \mid V \text{ has the local and modified distal Kato properties}\}$.
\end{definition}

\begin{proposition}\lb{low_bd} If $V \in \K$ has the local Kato property (in particular, if $V \in \K_0$ or $V \in \K_1$) and is real-valued (self-adjoint), then $\sigma(H)$ is bounded from below.
\end{proposition}

Recall that $V$ has the local Kato property if it can be approximated in $\Delta L^\infty$ by smooth, but not necessarily compactly supported, functions.

\begin{proof} When $V$ is a test function,
$$
\lim_{a \to +\infty} \sup_y \int_{\R^3} \frac {e^{-a|x-y|} |V(x)| \dd x} {|x-y|} = 0,
$$
and this property is stable under taking limits in $\K$, so it also holds for $V \in \K_0$. It follows that for any $V \in \K_0$ there exists $a_0$ such that if $\re \sqrt -\lambda > a_0$ then
$$
\sup_y \int_{\R^3} \frac {e^{-\sqrt{-\lambda}|x-y|} |V(x)| \dd x}  {|x-y|} < 4\pi
$$
or in other words $\|V R_0(\lambda)(x, y)\|_{(C_0)_y \M_x} < 1$.

If $V$ has the local Kato property, then the integral
$$
\int_{\R^3} \frac {e^{-a|x-y|} |V(x)| \dd x} {|x-y|}
$$
can be uniformly approximated for all $y$ and $a \geq 0$ by
$$
\int_{|x-y| \geq \epsilon} \frac {e^{-a|x-y|} |V(x)| \dd x} {|x-y|} <e^{-a\epsilon} \|V\|_\K \to 0
$$
as $a \to \infty$ and this vanishing is stable when taking the uniform limit for $\epsilon \to 0$.

Therefore, in this region $I+V R_0$ can be inverted in $\B(\M)$ or $\B(\K)$ by means of a Born series and can be used to obtain $R_V = R_0 (I+VR_0)^{-1}$.

Finally, since $V$ is real-valued, $\sigma(H) \subset \R$, so the above implies that $\sigma(H)$ is bounded from below.
\end{proof}

The invertibility of $I+VR_0$ can be restated equivalently in terms of $I+R_0V$ or, when $V$ is non-singular, $I+|V|^{1/2}R_0|V|^{1/2}\sgn V$ or $\sgn \ov V +|V|^{1/2}R_0|V|^{1/2}$, see \cite{sch}. The latter is a compact perturbation of a unitary (and self-adjoint, if $V$ is) operator on $L^2$.

We next prove a stronger spectral property, namely that $H$ has only finitely many negative eigenvalues.
\begin{proposition}\lb{finite} If $V \in \K_0$ is real-valued (self-adjoint), then $H$ has the same number of negative eigenvalues as the number of (real) eigenvalues of $V (-\Delta)^{-1} \in \B(\M) \cap \B(\K)$ in $(-\infty, -1)$. In particular, this number is finite.
\end{proposition}

Here we allow zero energy bound states, unlike in the rest of the paper.
%

\begin{proof}



Proposition \ref{low_bd} applies uniformly to
$$
H_t:=-\Delta+tV,
$$
so there exists $R_0$ such that when $t \in [0, 1]$ $H_t$ has no eigenvalue below $-R_0$.

Consider the contour $\mc C=\partial B(-R_0, R_0)$. The number of zeros of
$$
A_t(\lambda):=I+tA(\lambda)=I+tVR_0(\lambda)
$$
inside $\mc C$ does not change under homotopy, i.e.~on intervals in $t$ where $A_t(\lambda)$ has no zeros on $\mc C$.

Indeed, consider the set $\mc S=\{t \geq 0 \mid A_t(\lambda) \text{ has no zeros on } \mc C\}$. For each such $t$, since $V$ has the distal property, $A_t(\lambda)$ is continuous on $\mc C$ and analytic inside $\mc C$, and since there are no zeros $A_t^{-1}$ is uniformly bounded on $\mc C$. Moreover,
$$
\sup_{\lambda \in \C} \|A_{\tilde t}(\lambda)-A_t(\lambda)\| \leq \frac {\|V\|_\K}{4\pi} |\tilde t-t|.
$$

Rouch\'e's theorem (Theorem \ref{rouche}) then implies that the subsets of $\mc S$ on which the number of zeros inside $\mc C$ is constant are both open and closed, so the number of zeros is constant on the connected components, i.e.~ on each maximal subinterval.

Due to the self-adjointness of $-\Delta+V$ and to the uniform lower bound on the spectrum, zeros on $\mc C$ can only occur at $\lambda=0$. A zero of $A_t(0)=I+tVR_0(0) = I+tV(-\Delta)^{-1}$ corresponds to a negative eigenvalue $-1/t$ of $V(-\Delta)^{-1}$.

The interval $t \in [0, 1]$ corresponds to $-1/t \in (-\infty, -1]$. If $V \in \K_0$, then $V(-\Delta)^{-1}$ is compact, so it has finitely many eigenvalues in $(-\infty, -1]$. Thus the interval $(0, 1)$ can be decomposed into finitely many subintervals $(t_n, t_{n+1})$, where $t_0=0$, on each of which the number of negative eigenvalues of $H_t$ remains constant, plus the transition points $t_n$ that we need to understand.

Next, consider a modified contour $\tilde {\mc C} = \partial(B(-R, R) \setminus B(0, \epsilon))$ that follows $\mc C$ for the most part, but stays away from $0$. On this contour, $H_{t_n}$ has no zeros, because $\lambda=0$ is an isolated zero, see below. By the above it follows that $H_t$ will have the same number of eigenvalues as $H_{t_n}$ inside $\tilde {\mc C}$ for all $t$ in some small neighborhood of $t_n$.

Thus, either for $t<t_n$, for $t=t_n$, or for $t>t_n$, the total number of zeros inside $\mc C$ differs only according to what happens in $B(0, \epsilon)$. We are left with understanding what happens to the negative eigenvalue count in $B(0, \epsilon)$ when $A_{t_n}$ has a zero at $0$, for $t$ close to $t_n$.

With no loss of generality, let $t_n=1$ and suppose $A_1$ has a zero at $0$, i.e.~$A_1(0)$ is not invertible. As $A_1$ is a compact perturbation of the identity, $\ker A_1(0)$ is a finite-dimensional, closed subspace of $\M \cap \K$, on which all norms are equivalent. Since $V$ is self-adjoint, all zeros of $A_1(\lambda)$ are real.


Note that $R_0(0)$ is a dot product on $\ker A_1(0) \subset \M \cap \K$. Let $P_0$ be the projection on this subspace given by
$$
P_0 f = \sum_{k=1}^N \langle f, g_k \rangle f_k
$$
and let $P_1=I-P_0$. Here $f_k = -V R_0(0) f_k$, $1 \leq k \leq N$, and $f_k = - V g_k$, $g_k = R_0(0) f_k$, $g_k = -R_0(0) V g_k$ are the bound states of $H$ at energy $0$.

It is easy to see that $g_k \in L^\infty \cap \K^*$. Furthermore, $g_k \in \langle x \rangle^{-1} L^\infty$. To see this, write equation $g_k=-R_0(0)Vg_k$ as
$$
g_k = -R_0(0) V_1 g_k - R_0(0) V_2 g_k
$$
where $V_1=V \chi_{>R}(x)$ and $V_2=V \chi_{\leq R}(x)$ for some sufficiently large $R$. Then $I+R_0(0) V_1$ is invertible on $\langle x \rangle^{-1} L^\infty$ and $R_0(0) V_2 g_k \in \langle x \rangle^{-1} L^\infty$.

Note that $A_1(0)$ is invertible on $X_1=P_1 (\M \cap \K)$, hence so is $P_1 A_1(\lambda) P_1$ for $\lambda$ sufficiently close to $0$, with uniformly bounded inverse in $\B(\M) \cap \B(\K)$.

Next, we show that $\lambda=0$ has a punctured neighborhood containing no other zeros of $A_1$, using an analysis along the lines of \cite{ersc}. By Fechbach's Lemma \ref{fechbach}, for $\lambda$ close to $0$, $A_1(\lambda)$ is invertible if and only if
\be\lb{expr}
B_1(\lambda) = P_0 A_1(\lambda) P_0 - P_0 A_1(\lambda) P_1 (P_1 A_1(\lambda) P_1)^{-1} P_1 A_1(\lambda) P_0
\ee
is.

Next, suppose that (\ref{expr}) is not invertible. Since (\ref{expr}) is defined on a finite-dimensional vector space, this means there exists $f \in \Ker A_1(0)$, $f \ne 0$, $f=P_0 f$, such that
\be\lb{expr2}
\langle P_0 A_1(\lambda) P_0 f, g \rangle = \langle P_0 A_1(\lambda) P_1 (P_1 A_1(\lambda) P_1)^{-1} P_1 A_1(\lambda) P_0 f, g \rangle,
\ee
where $f=Vg$ is paired with $g=R_0(0)f$.

Since $P_0 A_1(0) = A_1(0) P_0 = 0$, $P_0 A_1(\lambda) = P_0 [A_1(\lambda)-A_1(0)]$, and $A_1(\lambda) P_0 = [A_1(\lambda)-A_1(0)] P_0$, rewrite the right-hand side as
$$
\langle P_0 [A_1(\lambda)-A_1(0)] P_1 [P_1 A_1(\lambda) P_1]^{-1} P_1 [A_1(\lambda)-A_1(0)] P_0 f, g \rangle.
$$
Note that
$$
\|R_0(\lambda)-R_0(0)\|_{\B(L^1, L^\infty)} \les \lambda^{1/2},\ \|R_0(\lambda)-R_0(0)\|_{\B(L^1, \K^*) \cap \B(\K, L^\infty)} \les 1.
$$
By approximating $V$ with potentials in $\K \cap \M$, we obtain
\be\lb{app1}
\|P_0 (A_1(\lambda)-A_1(0)) P_1 (P_1 A_1(\lambda) P_1)^{-1} P_1 (A_1(\lambda)-A_1(0)) P_0 f\| \les o(\lambda^{1/2}) \|f\|_{L^1}
\ee
as $\lambda \to 0$. At the same time, the left-hand side in (\ref{expr2}) admits the asymptotic expansion
$$
\langle P_0 A_1(\lambda) P_0 f, g \rangle = \langle (R_0(\lambda)-R_0(0)) f, f \rangle = c\lambda^{1/2} |\langle f, 1 \rangle|^2 + o(\lambda^{1/2}).
$$
Therefore $\langle f, 1\rangle=0$ and the corresponding $g=R_0(0)f \in L^2$ is an eigenfunction:
$$
g(x)=-\frac 1 {4\pi} \int_{\R^3} \frac {V(y) g(y) \dd y}{|x-y|} = -\frac 1 {4\pi} \int_{\R^3} \bigg( \frac {1}{|x-y|} - \frac 1 {|x|} \bigg) V(y) g(y) \dd y.
$$
In particular $g \in \langle x \rangle^{-2} L^\infty$ and $f \in \langle x \rangle^{-1} L^1 \cap \langle x \rangle^{-2} \K$. Then the left-hand side of (\ref{expr2}) has the asymptotic expansion
$$
c\lambda \int_{\R^3} \int_{\R^3} f(x) |x-y| f(y) \dd x \dd y + o(\lambda).
$$
For the right-hand side, note that
\be\lb{rhs1}
\frac {e^{i\sqrt \lambda |x-y|} - e^{i\sqrt\lambda|x|}}{|x|} \les \lambda^{1/2} \frac {|y|}{|x|},\ \bigg\|\int_{\R^3} \frac {e^{i\sqrt \lambda |x-y|} - e^{i\sqrt\lambda|x|}}{|x|} f(y) \dd y\bigg\|_{|x|^{-1} L^\infty} \les \lambda^{1/2} \|f\|_{\langle x \rangle^{-1} L^1}.
\ee
Since $\langle f, 1\rangle = \langle V g, 1 \rangle=0$, one can subtract the term
$$
\frac 1 {4\pi} \int_{\R^3} \frac {e^{i\sqrt\lambda|x|}-1}{|x|} f(y) \dd y = 0
$$
from
$$
[R_0(\lambda)-R_0(0)] f = \frac 1 {4\pi} \int_{\R^3} \frac {e^{i\sqrt\lambda|x-y|}-1}{|x-y|} f(y) \dd y
$$
then in light of (\ref{rhs1}) we can replace this term by $\frac {e^{i\sqrt\lambda|x-y|}-1}{|x|} f(y)$. Furthermore
$$
\frac {e^{i\sqrt\lambda|x-y|}-1}{|x-y|} - \frac {e^{i\sqrt\lambda|x-y|}-1}{|x|} \les \lambda^{1/2} \frac {|y|}{|x|},\ \bigg\|\int_{\R^3} \bigg(\frac {e^{i\sqrt\lambda|x-y|}-1}{|x-y|} - \frac {e^{i\sqrt\lambda|x-y|}-1}{|x|}\bigg) f(y) \dd y\bigg\|_{|x|^{-1} L^\infty} \les \lambda^{1/2} \|f\|_{\langle x \rangle^{-1} L^1}.
$$
Consequently
$$
\|(A_1(\lambda)-A_1(0)) P_0 f\|_{L^1} \les \lambda^{1/2} \|f\|_{\langle x \rangle^{-1} L^1}
$$
and for the right-hand side in (\ref{expr2}) we obtain
$$
\|P_0 (A_1(\lambda)-A_1(0)) P_1 (P_1 A_1(\lambda) P_1)^{-1} P_1 (A_1(\lambda)-A_1(0)) P_0 f\| \les \lambda \|f\|_{\langle x \rangle^{-1} L^1}.
$$
If $V \in \langle x \rangle^{-1} L^1$, since $\langle f, 1 \rangle = 0$ and
$$
\frac {e^{i\sqrt \lambda |x-y|}-1-i\sqrt\lambda|x-y|}{|x-y|} \les \lambda |x-y|,
$$
a better estimate is
$$
\|(A_1(\lambda)-A_1(0)) P_0 f\|_{L^1} \les \lambda \|f\|_{\langle x \rangle^{-1} L^1}.
$$
Consequently, approximating $V$ with potentials in this better class we obtain
\be\lb{app2}
\|P_0 (A_1(\lambda)-A_1(0)) P_1 (P_1 A_1(\lambda) P_1)^{-1} P_1 (A_1(\lambda)-A_1(0)) P_0 f\| \les o(\lambda) \|f\|_{\langle x \rangle^{-1} L^1}.
\ee

This leads to a contradiction in (\ref{expr2}): $\int_{\R^3} \int_{\R^3} f(x)|x-y|f(y) \dd x \dd y = 0$ and since $\langle f, 1 \rangle = 0$ this means $\|f\|_{\dot H^{-1}} = 0$, hence $f=0$, which contradicts our original assumption that $f \ne 0$.

In the proof of the next theorem we show the absence of positive energy bound states, implying that $\lambda=0$ is also an isolated zero on the boundary. For the present discussion, only what happens in the left half-plane is relevant.


Without loss of generality, let $t_n=1$. Then $A_t(\lambda)$
for $|t-1|=\epsilon<<1$ and $|\lambda|<\epsilon^2$ ($t$ close to $1$ and $\lambda$ close to $0$) is invertible if and only if
\be\lb{full}
P_0 A_t(\lambda) P_0 - P_0 A_t(\lambda) P_1 (P_1 A_t(\lambda) P_1)^{-1} P_1 A_t(\lambda) P_0
\ee
is invertible on $\Ker A_1(0)$. Recall $P_0$ is the projection on $\Ker A_1(0)$ and $P_1 = I-P_0$.

For $\lambda=0$, the first term in (\ref{full}) is non-degenerate whenever $t \ne 1$, being equivalent to the $\dot H^{-1}$ dot product, up to sign:
$$
|\langle P_0 A_t(0) P_0 f, \tilde g \rangle| = |t-1| \langle f, \tilde f \rangle_{\dot H^{-1}}. 
$$
Its inverse is of size $|t-1|^{-1}$. Since the second term is quadratic, it is of size $o(t-1)$ as $t$ approaches $1$, so for $t$ sufficiently close to $1$ the expression is invertible.

Furthermore, for each $|t|<1$ this quadratic form is, uniformly in $\lambda$, non-degenerate, because it is bounded from below by some multiple of the $\dot H^{-1}$ dot product. Thus, when $t<1$ is close to $1$, $A_t$ has the same number of negative zeros as $A_1$ in a small punctured neighborhood of $\lambda=0$, namely no zeros.

When $t>1$ is close to $1$, by Rouch\'e's Theorem \ref{rouche} $A_t$ has a constant number of zeros in a small neighborhood of $0$, and by the above it does not have a zero at $0$.

Consider the first term in (\ref{full}), $P_0 A_t(\lambda) P_0 = P_0 (A_t(\lambda)-A_1(0)) P_0$. This can be written as
$$
\langle P_0 (A_t(\lambda)-A_1(0)) P_0 f, \tilde g \rangle = \langle (-\Delta)^{-1/2} f, [I-t (-\Delta)^{1/2} (-\Delta-\lambda)^{-1} (-\Delta)^{1/2}] (-\Delta)^{-1/2} \tilde f \rangle,
$$
where $(-\Delta)^{-1/2}$ maps $\Ker A_1(0)$ injectively into $L^2$. The quadratic form
$$
I-t (-\Delta)^{1/2} (-\Delta-\lambda)^{-1} (-\Delta)^{1/2}
$$
is represented by a hermitian matrix. In order for it to be singular, $1/t$ has to be an eigenvalue of the hermitian matrix
$$
S_\lambda=(-\Delta)^{1/2} (-\Delta-\lambda)^{-1} (-\Delta)^{1/2}
$$
on $(-\Delta)^{-1/2} \Ker A_1(0)$.

Note that $0<S_\lambda \leq S_0=I$ and $S_\lambda$ is a strictly monotone function of $\lambda$, when $\lambda \leq 0$. By the minmax principle the eigenvalues of $S_\lambda$ are monotone, meaning that if we parametrize them as $1 \geq \eta_1(\lambda) \geq \ldots \geq \eta_N(\lambda) > 0$ for each $\lambda$ then $\eta_k(\lambda)$ is a strictly monotone function of $\lambda$ for each $k$.

Consequently, for each $t \geq 1$ close to $1$ and each $1 \leq k \leq N$, $1/t=\eta_k(\lambda)$ for exactly one value of $\lambda$. Thus, the first term in (\ref{full}) has exactly $N$ zeros.

The same is true for the full expression (\ref{full}), but a deeper analysis is required. Further decompose $\Ker A_1(0)$ by considering the subspace of functions $f$ orthogonal to $1$, which correspond to eigenstates: $f=-Vg$ where $g$ is an eigenstate. We take the orthogonal complement of this subspace with respect to $(-\Delta)^{-1}$, which defines a dot product on $\Ker A_1(0)$. The orthogonal complement has dimension at most $1$, consisting of the canonical resonance.

Let $P_0=Q_1+Q_2$, where $Q_1 f = \langle f, g_1 \rangle f_1$ corresponds to the canonical resonance and $Q_2 f = \sum_{k=2}^N \langle f, g_k \rangle f_k$ to the eigenstates. We only examine the most complicated case, when $Q_1 \ne 0$ and $Q_2 \ne 0$.

Let $\lambda=-k^2 \leq 0$. We estimate the second term in (\ref{full}). In particular, we need to deal with factors of $A_t(0)-A_1(0)$. Note that $A_t(0)-A_1(0)=(t-1)V R_0(0)$, so $P_0 [A_t(0)-A_1(0)] = [A_t(0)-A_1(0)] P_0 = (1-t) P_0$. Due to the presence of $P_1$, the contribution of such factors to the second term in (\ref{full}) cancels.

Due to estimates (\ref{app1}) and (\ref{app2}), the second term in (\ref{full}) is then of size
\be\lb{err}
E=P_0 (o(k) Q_1 + o(k^2) Q_2).
\ee

Consider the Taylor expansion of $(-\Delta-\lambda)^{-1}$:
$$
(-\Delta+k^2)^{-1}(x, y) = \frac 1 {4\pi} \frac {e^{-k |x-y|}}{|x-y|} = \frac 1 {4\pi} \bigg(\frac 1 {|x-y|} - k 1 + o(k)\bigg) = \frac 1 {4\pi} \bigg(\frac 1 {|x-y|} - k 1 + \frac {k^2} 2 |x-y| + o(k^2|x-y|)\bigg).
$$
Here $1$ is a rank-one non-negative operator and $\ds\frac 1 {4\pi} |x-y|$ is a negative operator on the space of $\langle x \rangle^{-1} L^1$ functions orthogonal to $1$ (i.e.~whose integral is zero), where it is the integral kernel of $-(-\Delta)^{-2}$.

This produces the following asymptotic expansion of the first term in (\ref{full}) $P_0 A_t(\lambda) P_0 = P_0 [A_t(\lambda)-A_1(0)] P_0$:
$$\begin{aligned}
P_0 A_t(\lambda) P_0 &= \frac 1 {4\pi} \bigg((t-1) P_0 V \frac 1 {|x-y|} P_0 - (tk) Q_1 V 1 Q_1 + Q_1 o(k) Q_1 + \frac {tk^2} 2 Q_2 V |x-y| Q_2 +\\
&+ t Q_1 V \frac {e^{-k|x-y|}-1+k|x-y|}{|x-y|} Q_2 + t Q_2 V \frac {e^{-k|x-y|}-1+k|x-y|}{|x-y|} Q_1 + Q_2 o(k^2) Q_2\bigg).
\end{aligned}$$

We have written explicitly the most problematic two error terms. To bound them, we have to evaluate the integrals
\be\lb{ex}
\int_{\R^3} \int_{\R^3} f_1(x) \frac {e^{-k|x-y|}-1+k|x-y|}{|x-y|} f_j(y) \dd x \dd y,
\ee
where $2 \leq j \leq N$. Note that in the integral
$$
\int_{\R^3} \int_{\R^3} f_1(x) |x-y| f_j(y) \dd x \dd y
$$
one can replace $|x-y|$ by $|x-y|-|x|$, since $\langle f_j, 1 \rangle = 0$. Then $||x-y|-|x||\leq |y|$ and since $f_j \in \langle x \rangle^{-1} L^1$ the integral is well-defined, but not absolutely convergent.

Similarly, let $F(t)=\frac {e^{-kt}-1+kt}{t}$; then $F(|x-y|)-F(|x|) \les k^2 |y|^2$. As before, we can subtract $F(|x|)$ in (\ref{ex}) for free, since $\langle f_j, 1 \rangle=0$, and then the integral of the difference is of size $O(k^2)$, since $f_j \in \langle x \rangle^{-2} \K$.

Grouping most of the error terms together and noting that $Q_1 o(k) Q_1$, $Q_2 O(k^2) Q_1$, and $Q_2 o(k^2) Q_2$ satisfy the bound (\ref{err}), we then approximate (\ref{full}) and its first term $P_0(A_t(\lambda)-A_1(0))P_0$ by
$$\begin{aligned}
B_{t, \lambda} &= \frac 1 {4\pi} \bigg((t-1) P_0 V \frac 1 {|x-y|} P_0 - (t k) Q_1 V 1 Q_1 + \frac {t k^2} 2 Q_2 V |x-y| Q_2 + Q_1 O(k^2) Q_2 \bigg) \\
&= (t-1) B_0 - tk B_1 - tk^2 B_2 + Q_1 O(k^2) Q_2 \\
&= Q_1 [(t-1)I-tkVB_1] Q_1 + Q_2 [(t-1)I - tk^2 VB_2 ] Q_2 + Q_1 O(k^2) Q_2.
\end{aligned}$$
Here $B_0$ is the matrix representation of $\ds (-\Delta)^{-1} = V \frac 1 {4\pi|x-y|}$ as a quadratic form on $\Ker A_1(0)$, which we can take to be the identity matrix by choosing an appropriate basis, while $B_1=Q_1 V B_1 Q_1$ is the matrix representation of $\ds P_0 V \frac 1 {4\pi} P_0$ and $B_2=Q_2 V B_2 Q_2$ is the matrix representation of $\ds-\frac 1 {8\pi} Q_2 V |x-y| Q_2 = \frac 1 2 (-\Delta)^{-2}$ on $Q_2 \Ker A_1(0)$.

One can represent $B_{t, \lambda}$ as a block matrix:
$$
\begin{pmatrix}
(t-1)I-tkB_1 & O(k^2) \\
0 & (t-1)I-tk^2 B_2
\end{pmatrix} = \begin{pmatrix} B_{11} & B_{12} \\ 0 & B_{22} \end{pmatrix}.
$$

We first examine the invertibility of the diagonal entries. Note that $B_1$ has exactly one eigenvalue $\lambda_1$ and $B_2$ has $N-1$ eigenvalues $\lambda_2, \ldots, \lambda_N$, all positive, since $B_1$ and $B_2$ are hermitian and positive definite.

For $t>1$ close to $1$, $B_{t, \lambda}$ has exactly $N$ zeros, one for $\ds \frac {t-1}{tk} = \lambda_1$ and $N-1$ zeros for $\ds \frac {t-1}{tk^2} = \lambda_j$, $2 \leq j \leq N$. In terms of $t-1=\epsilon<<1$, the zeros are at $k \sim \epsilon$ and $k \sim \epsilon^{1/2}$, meaning $\lambda \sim -\epsilon^2$ and $\lambda \sim -\epsilon$.

The inverse of $B_{t, \lambda}$ is
$$
B_{t, \lambda}^{-1} = \begin{pmatrix} B_{11}^{-1} & -B_{11}^{-1} B_{12} B_{22}^{-1} \\ 0 & B_{22}^{-1} \end{pmatrix}.
$$
When $B_{t, \lambda}$ is invertible, the inverses of the diagonal terms are of size
$$
Q_1 O((t-1-tk\lambda_1)^{-1}) Q_1
$$
and
$$
Q_2 O(\sup_{2 \leq j \leq N} |t-1-tk^2\lambda_j|^{-1}) Q_2
$$
and the off-diagonal term in the inverse has size
$$
Q_1 O((t-1-tk\lambda_1)^{-1}) O(k^2) O(\sup_{2 \leq j \leq N} |t-1-tk^2\lambda_j|^{-1}) Q_2.
$$

Consider a contour (a small square, for example) that cuts the real axis at $\lambda=0$ and $\lambda=-k^2=-C\epsilon$. For $C$ fixed and sufficiently large, all $N$ zeros are inside the contour.

On the contour, near $\lambda=0$ (for $|\lambda|<\epsilon^2$), the inverses of the diagonal terms are of size $\epsilon^{-1}$ and the off-diagonal term is of size $1$. Consequently the product between the error term $E$ (\ref{err}) and $B_{t, \lambda}^{-1}$ is of size $E B_{t, \lambda}^{-1} \les o(1)$. Away from $0$ the inverses of the diagonal terms are of size $Q_1 O(k^{-1}) Q_1$ and $Q_2 O(\epsilon^{-1}) Q_2$, while the off-diagonal term is of size $Q_1 O(k \epsilon^{-1}) Q_2 = Q_1 O(k^{-1}) Q_2$, as $k\epsilon^{-1} \les k^{-1}$. Since $\epsilon^{-1} \les k^{-2}$, again $E B_{t, \lambda}^{-1} \les o(1)$.

%

By Rouch\'e's Theorem \ref{rouche}, then, for sufficiently small $\epsilon$, the full expression (\ref{full}) has the same number $N$ of zeros inside the contour as the approximation $B_{t, \lambda}$.

Thus, the number of negative eigenvalues of $H_t$ grows by no more than the number of zeros at each transition point $t_n$, considered with multiplicity.

Finally, note that $t$ is a zero for $A_t(0)=I+tV (-\Delta)^{-1}$ exactly when $-t^{-1}$ is an eigenvalue of $V (-\Delta)^{-1}$ and $t \in [0, 1]$ corresponds to real eigenvalues below $-1$. If $V \in \K_0$, then $V (-\Delta)^{-1}$ is compact, so it has finitely many eigenvalues in the interval $(-\infty, -1]$ and more generally at any positive distance away from $0$.
\end{proof}

We next prove the absence of positive eigenvalues, using the following result about bound states along the positive real axis.
\begin{lemma}\lb{more_decay} Suppose that $V \in \K_0$ and $f \in \ker(I+VR_0(0))$. If $\langle f, 1 \rangle=\int_{\R^3} f = 0$, then $f \in \langle x \rangle^{-1}\M \cap \langle x \rangle^{-1}\K$ and $g=R_0(0)f \in \langle x \rangle^{-1} L^\infty \cap \langle x \rangle^{-1} \K^*$.

More generally, suppose that $V \in \K_0$ and $f \in \Ker(I+R_0(\eta_0^2))$, $\eta_0 \in \R$. If for each $x \in \supp V$ $\int_{\R^3} e^{i\eta_0 |x-y|} f(y) \dd y = 0$, then $f \in \langle x \rangle^{-1}\M \cap \langle x \rangle^{-1}\K$ and $g=R_0(\eta_0^2)f \in \langle x \rangle^{-1} L^\infty \cap \langle x \rangle^{-1} \K^*$.
\end{lemma}
This is partly already known, following \cite{ersc} and others.
\begin{proof} In both cases, since $V \in \K_0$, the kernel has finite dimension and $f \in \M \cap \K$. Write the identity $f=-VR_0(\eta_0^2)f$ as
$$
f(x)=\frac {V(x)}{4\pi} \int_{\R^3} e^{i\eta_0|x-y|} \bigg(\frac 1 {|x|} - \frac 1 {|x-y|}\bigg) f(y) \dd y,
$$
where $\lambda=\eta^2$, and note that $|\frac 1 {|x|} - \frac 1 {|x-y|}| \leq \frac {|y|}{|x-y||x|}$.

Since $V$ has the modified distal property, write it as $V=V_1+V_2$, where $V_1$ has compact support and $\|V_2\|_\K < 4\pi$ is small. Then
$$
|x| f(x) = -\frac {V_1(x)}{4\pi} \int_{\R^3} \frac {e^{i\eta_0|x-y|}}{|x-y|} f(y) \dd y + \frac {V_2(x)}{4\pi} \int_{\R^3} \frac {e^{i\eta|x-y|}}{|y|} \bigg(\frac 1 {|x|} - \frac 1 {|x-y|}\bigg) |y| f(y) \dd y.
$$
The first term is compactly supported and the second represents a contraction, so $|x| f(x) \in \M \cap \K$. The conclusion about $g$ follows by rewriting its defining integral in the same fashion.
\end{proof}

\begin{proposition}\lb{no_pos} If $V \in \K_0$ is real-valued (self-adjoint) and has the local and modified distal Kato properties, then $H=-\Delta+V$ has no embedded resonances in $(0, \infty)$ and embedded eigenvalues in $(0, \infty)$ are discrete. If in addition $V \in L^{3/2, \infty}$ or $V \in \dot W^{-1/4, 4/3}$, then there are no embedded eigenvalues.
\end{proposition}
Note that at zero energy there can be both resonances and eigenvalues, which in the rest of the paper we explicitly exclude by Assumption \ref{assum}. The present result does not depend on Assumption \ref{assum}.

The absence of embedded resonances in $(0, \infty)$ was proved in \cite{goldberg}, in the sense that any embedded resonance must actually be an eigenvalue. Under extra assumptions on the potential $V$, \cite{ioje} and \cite{kota} prove there are no embedded eigenvalues either.
\begin{proof} 


Let us prove that any zero on the boundary is isolated along the boundary and in $\C \setminus [0, \infty)$ both, meaning that it is isolated on the whole closed sheet/half plane.

Consider the operator $A_1(\eta_k^2)$ on $\ker A_1(\eta_0^2)$. Since $V \in \K_0$, $\ker A_1(\eta_0^2)$ is finite-dimensional, so its unit sphere is compact. Let $P_0$ be the projection on $\Ker A_1(\eta_0^2)$ and $P_1=I-P_0$. For $A_1(\eta_k^2)$ not to be invertible, by Fechbach's Lemma \ref{fechbach}
$$
\tilde B_{1, \lambda_k} = P_0 A_1(\eta_k^2) P_0 - P_0 A_1(\eta_k^2) P_1 (P_1 A_1(\eta_k^2) P_1)^{-1} P_1 A_1(\eta_k^2) P_0
$$
must not be invertible. Hence there must exist $f_k \in \ker A_1(\eta_0^2)$, $f_k \ne 0$, such that $\tilde B_{1, \lambda_k} f_k=0$.

Suppose $\eta_k \to \eta_0$ and $f_k \to f_0$, $f_0 \ne 0$, in $\ker A_1(\eta_0^2)$. Then
$$
0=P_0[A_1(\eta_k^2)-A_1(\eta_0^2)] P_0f_k + o(\eta_k-\eta_0) = P_0 V [R_0(\eta_k^2)-R_0(\eta_0^2)] f_k + o(\eta_k-\eta_0) = \frac i {4\pi} (\eta_k-\eta_0) P_0 V e^{i\eta_0|x-y|} f_k + o(\eta_k-\eta_0).
$$
Therefore $P_0 V e^{i\eta_0|x-y|} f_k \to 0$, so
\be\lb{formula1}
P_0 V e^{i\eta_0|x-y|} f_0 = 0.
\ee

If $\eta_0 \ne 0$, then
$$
e^{i\eta_0|x-y|} = \partial_\eta R_0(\eta_0^2) = 2\eta_0 R_0(\lambda_0)^2.
$$
Pairing $P_0 V e^{i\eta_0|x-y|} f_0 = 0$ with $g_0=R_0(\lambda_0 + i0)f_0$, we get that $\langle f_0, R_0(\lambda_0+i0)^2 f_0 \rangle=0$.

At this point, note that $g_0 \in L^\infty \cap \K^*$ can be paired with $V g_0 \in \K \cap L^1$ and the outcome is real-valued, since $V$ is self-adjoint. By the usual bootstrap argument \cite{Agm}, this can be written as
\be\lb{temp}
0 = \Im \langle R_0(\lambda_0+i0) V g_0, V g_0 \rangle = C(\eta_0) \int_{S^2} |\widehat {V g_0}(\eta_0 \omega)|^2 \dd \omega.
\ee
Since the integral is non-negative and equals $0$, the integrand must vanish. Consequently, $g_0=R_0(\lambda+i0) V g_0$ if and only if $g_0=R_0(\lambda-i0) V g_0$. Also, $R_0(\lambda+i0) f_0 = R_0(\lambda-i0) f_0$.

By Lemma \ref{more_decay}, $f_0 \in \langle x \rangle^{-1} \M \cap \langle x \rangle^{-1} \K$ and $g_0 \in \langle x \rangle^{-1} L^\infty \cap \langle x \rangle^{-1} \K^* \subset L^2$. Then
$$
\|g_0\|_{L^2}^2 = \langle g_0, g_0 \rangle = \langle R_0(\lambda+i0) f_0, R_0(\lambda+i0) f_0\rangle = \langle R_0(\lambda-i0) f_0, R_0(\lambda+i0) f_0\rangle = \langle f_0, R_0(\lambda+i0)^2 f_0 \rangle = 0.
$$
Then $g_0=0$, so $f_0=0$, leading to a contradiction. Hence any nonzero zero $\lambda_0 \ne 0$ along the boundary is isolated.

The analysis for $\eta_0=0$ is different, following \cite{ersc}. Again, suppose that $\eta_k \to \eta_0$ and $f_k \to f_0$ in $\ker A_1(0)$, where $\tilde B_{1, \lambda_k}(\eta_k^2) f_k = 0$.

A priori, $f_0 \in \M \cap \K$. Repeating the same reasoning as in (\ref{formula1}), we obtain that $P_0 V \int_{\R^3} f_0 = 0$ and pairing with $g_0 =R_0(0) f_0 \in L^\infty \cap \K^*$ it follows that
\be\lb{formula2}
\langle f_0, 1\rangle = \int_{\R^3} f_0 = 0.
\ee

Then, by Lemma \ref{more_decay}, $f_0 \in \langle x \rangle^{-1} \M \cap \langle x \rangle^{-1} \K$ and $g_0 \in \langle x \rangle^{-1} L^\infty \cap \langle x \rangle^{-1} \K^*$. Next,
$$
0=P_0 V[R_0(\eta_k^2)-R_0(0)]f_k+o(\eta_k^2)=P_0(\eta_k V \langle f_k, 1\rangle + \eta_k^2 V\partial_\lambda R_0(0) f_k) + o(\eta_k^2).
$$
Since $g_0 \in \langle x \rangle^{-1} L^\infty \cap \langle x \rangle^{-1} \K^* \subset L^2$, and $\langle g_0, V\rangle = \langle f_0, 1 \rangle = 0$, pairing $g_0$ with the previous expression we get
$$
\eta_k^2 \langle f_0, R_0^2(0) f_k \rangle + o(\eta_k^2)=0
$$
and by passing to the limit $\langle f_0, R_0^2(0) f_0 \rangle=0$. However, $g_0=R_0(0) f_0 \in L^2$, so this is equivalent to $\|g_0\|_{L^2}^2 = \langle R_0(0) f_0, R_0(0) f_0 \rangle = 0$. Then $f_0=-Vg_0=0$, producing a contradiction.

This finishes the proof that $\lambda_0=0$ must be an isolated zero, if it is a zero of $A_1$.



%
%

Next, if $\eta_0 \ne 0$, by (\ref{temp}) $\widehat f_0 = \widehat {V g_0} = 0$ on $\eta_0 S^2$, so by Corollary 13 in \cite{golsch} $g_0 \in L^2$. Thus, there can be no embedded resonances, only eigenvalues. Then, by the results of \cite{kota}, under the additional assumption that $V \in L^{3/2, \infty}$ or $V \in \dot W^{-1/4, 4/3}$, embedded eigenvalues are also excluded.

\end{proof}

\begin{proposition}\lb{cond_wiener} Consider $V \in \K$ and let $S(t)(x, y)= \chi_{t \geq 0}\frac 1 {4 \pi t} V(x) \delta_t(|x-y|)$. If $V$ satisfies the distal Kato condition, then $S$ decays at infinity:
$$
\lim_{R \to \infty} \|\chi_{\rho \geq R} S(\rho)\|_{L^\infty_y \M_x L^1_\rho} = 0.
$$
If $V \in \K_0$, then $S^N$ is continuous under translation when $N=2$:
$$
\lim_{\epsilon \to 0} \|S^2(\rho+\epsilon) - S^2(\rho)\|_{L^\infty_y \M_x L^1_\rho} = 0.
$$
\end{proposition}
\begin{proof}
The decay at infinity condition for Wiener's theorem is exactly equivalent to the distal Kato condition.

Likewise, if $\tilde V$ is an approximation of $V$,
$$
\|S*S-\tilde S*\tilde S\|_{L^\infty_y \K_x L^1_\rho} \leq \|S-\tilde S\|_{L^\infty_y \K_x L^1_\rho} (\|S\|_{L^\infty_y \K_x L^1_\rho} + \|\tilde S\|_{L^\infty_y \K_x L^1_\rho})
$$
and in general
$$
\|S^N-\tilde S^N\|_{L^\infty_y \K_x L^1_\rho} \leq \|S-\tilde S\|_{L^\infty_y \K_x L^1_\rho} (\|S\|_{L^\infty_y \K_x L^1_\rho}^{N-1} + \|S\|_{L^\infty_y \K_x L^1_\rho}^{N-2} \|\tilde S\|_{L^\infty_y \K_x L^1_\rho} + \ldots + \|\tilde S\|_{L^\infty_y \K_x L^1_\rho}^{N-1}).
$$
The same is true for translations of $S^N$ and $\tilde S^N$. Thus, if the continuity condition holds for a sequence of such $\tilde S \to S$, for some fixed $N$, then it also holds for $S$.

Furthermore, $\|S-\tilde S\|_{L^\infty_y \K_x L^1_\rho} \leq \|V-\tilde V\|_\K$. Hence, if the continuity condition holds for all test functions for some fixed $N$, it will follow for all $V \in \K_0$ for the same $N$.

Finally, the continuity condition holds for all test functions, by \cite{becgol}, when $N=4$ and, by Lemma \ref{second_cond}, when $N=2$.
\end{proof}

\section*{Acknowledgments}
M.B.~has been supported by the NSF grant DMS--1700293 and by the Simons Collaboration Grant No.~429698. M.B.~would also like to thank the University at Albany and the University of Cincinnati for hosting him while working on some parts of this paper.

Both authors would like to thank Shijun Zheng for the helpful conversations.

\end{document}